\documentclass[12pt]{amsart}
\usepackage{a4wide}
\usepackage{times}
\usepackage{mathtools, amssymb}
\usepackage{graphicx,xspace}
\usepackage{epsfig}
\usepackage{dsfont}
\usepackage[usenames,dvipsnames]{xcolor}
\usepackage{tikz}
\usepackage[T1]{fontenc}
\usepackage[utf8]{inputenc}
\usepackage{array,multirow} %
\usepackage{bm}
\usepackage{cancel}

\usepackage{hyperref,pifont}
\usepackage{todonotes}
\usepackage{dsfont}

\usepackage[capitalize]{cleveref}

\theoremstyle{plain}

\newtheorem{lemma}{Lemma}[section]
\newtheorem{theorem}[lemma]{Theorem}

\newtheorem{proposition}[lemma]{Proposition}
\newtheorem{definition}[lemma]{Definition}

\theoremstyle{remark}
\newtheorem{remark}[lemma]{Remark}

\newtheorem{observation}[lemma]{Observation}

\def\Sn{\mathfrak{S}}   %

\def\R{\mathbb{R}}

\def\density{{\sf Dens}}
\def\cograph{{\sf Cograph}} 
\def\Cograph{{\sf Cograph}}
\def\One{\mathtt{1}}
\def\Zero{\mathtt{0}}

\def\formerell{j}

\def\even{\mathrm{even}}
\def\odd{\mathrm{odd}}
\def\proba{\mathbb{P}}
\def\esper{\mathbb{E}}

\DeclareMathOperator{\SubGraph}{SubGraph}
\DeclareMathOperator{\Sample}{Sample}
\DeclareMathOperator{\Dec}{Dec}
\DeclareMathOperator{\rank}{rk}

\def\dbox{\delta_{\Box}}
\def\SpaceGraphon{\widetilde{\mathcal W_0}}
\newcommand{\Exc}{ \scalebox{1.1}{$\mathfrak{e}$}}

\newcommand\restr[2]{{%
		\left.\kern-\nulldelimiterspace %
		#1 %
		\right|_{#2} %
	}}

\title[Limit of random cographs]{Random cographs: Brownian graphon limit\\
and asymptotic degree distribution}

\author[F. Bassino]{Frédérique Bassino}
       \address[FB]{Université Paris 13, Sorbonne Paris Cité, LIPN, CNRS UMR 7030, F-93430 Villetaneuse, France}
       \email{bassino@lipn.univ-paris13.fr}

 \author[M. Bouvel]{Mathilde Bouvel}
       \address[MB]{Institut für Mathematik, Universität Zürich, Winterthurerstr. 190, CH-8057 Zürich, Switzerland, and \\
       Université de Lorraine, CNRS, Inria, LORIA, F 54000 Nancy, France}
       \email{mathilde.bouvel@loria.fr}

       \author[V. Féray]{Valentin Féray}
       \address[VF]{Université de Lorraine, CNRS, IECL, F 54000 Nancy, France}
       \email{valentin.feray@univ-lorraine.fr}
       
 \author[L. Gerin]{Lucas Gerin}
       \address[LG]{CMAP, \'Ecole Polytechnique, CNRS, Route de Saclay, 91128 Palaiseau Cedex, France}
       \email{gerin@cmap.polytechnique.fr}

  \author[M. Maazoun]{Mickaël Maazoun}
   \address[MM]{École Normale Supérieure de Lyon, UMPA UMR 5669 CNRS, 46 allée d’Italie, 69364 Lyon Cedex 07, France}
   \email{mickael.maazoun@ens-lyon.fr}

 \author[A. Pierrot]{Adeline Pierrot}
 \address[AP]{Université Paris-Saclay, CNRS, Laboratoire Interdisciplinaire des Sciences du Numérique, 91400, Orsay, France}
       \email{adeline.pierrot@lri.fr}



\begin{document}

\begin{abstract}
We consider uniform random cographs (either labeled or unlabeled) of large size.
Our first main result is the convergence towards a Brownian limiting object in the space of graphons.
We then show that the degree of a uniform random vertex in a uniform cograph is of order $n$,
and converges after normalization to the Lebesgue measure on $[0,1]$.
We finally analyze the vertex connectivity (\emph{i.e.} the minimal number of vertices whose removal disconnects the graph) of random connected cographs, 
and show that this statistics converges in distribution without renormalization. 
Unlike for the graphon limit and for the degree of a random vertex,
the limiting distribution of the vertex connectivity is different in the labeled and unlabeled settings.

Our proofs rely on the classical encoding of cographs via cotrees.
We then use mainly combinatorial arguments, including the symbolic method and singularity analysis.
\end{abstract}

\maketitle


\section{Introduction}

\subsection{Motivation}

Random graphs are arguably the most studied objects at the interface of combinatorics and probability theory. 
One aspect of their study consists in analyzing a uniform random graph of large size $n$ in a prescribed family, 
\emph{e.g.} perfect graphs \cite{RandomPerfect}, planar graphs \cite{NoyPlanarICM},
 graphs embeddable in a surface of given genus \cite{RandomGraphsSurface},
  graphs in subcritical classes \cite{SubcriticalClasses}, 
  hereditary classes \cite{GraphonsEntropy} 
  or addable classes \cite{ConjectureAddable,ProofConjectureAddable}. 
The present paper focuses on uniform random \emph{cographs} (both in the labeled and unlabeled settings).
\medskip

Cographs were introduced in the seventies by several authors independently, see \emph{e.g.}~\cite{Seinsche} 
and further references on the Wikipedia page~\cite{Wikipedia}. 
They enjoy several equivalent characterizations. Among others, cographs are  
\begin{itemize}
 \item the graphs avoiding $P_4$ (the path with four vertices) as an induced subgraph;
 \item the graphs which can be constructed from graphs with one vertex by taking disjoint unions and joins;
 \item the graphs whose modular decomposition does not involve any prime graph; 
 \item the inversion graphs of separable permutations.
\end{itemize}

Cographs have been extensively studied in the algorithmic literature. 
They are recognizable in linear time~\cite{CorneilLinear,Habib,Bretscher} 
and many computationally hard problems on general graphs are solvable in polynomial time when restricted to cographs; see \cite{corneil} and several subsequent works citing this article. 
In these works, as well as in the present paper, 
a key ingredient is the encoding of cographs by some trees, called \emph{cotrees}. 
These cotrees witness the construction of cographs using disjoint unions and joins (mentioned in the second item above).

To our knowledge, cographs have however not been studied from a probabilistic perspective so far.
Our motivation to the study of random cographs comes from our previous work~\cite{Nous1,Nous2,SubsClosedRandomTrees,Nous3} 
which exhibits a Brownian limiting object for separable permutations (and various other permutation classes). 
The first main result of this paper (\cref{th:MainTheorem}) is the description of a Brownian limit for cographs. 
Although cographs are the inversion graphs of separable permutations, 
this result is not a consequence of the previous one on permutations: 
indeed the inversion graph is not an injective mapping,
hence a uniform cograph is not the cograph of a uniform separable permutation. 

Our convergence result holds in the space of \emph{graphons}. 
Graphon convergence has been introduced in~\cite{IntroducingGraphons} 
and has since then been a major topic of interest in graph combinatorics -- see \cite{LovaszBook} for a broad perspective on the field. 
The question of studying graphon limits of uniform random graphs
(either labeled or unlabeled) in a given class is raised by Janson
in \cite{JansonGraphLimits} (see Remark 1.6 there).
Some general results have been recently obtained for hereditary\footnote{A class of graphs 
is hereditary if any induced subgraph of a graph in the class is in the class as well.}
classes in \cite{GraphonsEntropy}.
However, these results (in particular Theorem 3 in \cite{GraphonsEntropy})
do not apply to cographs, 
since the class of cographs contain $e^{o(n^2)}$ graphs of size $n$.

The graphon limit of cographs found here, which we call the \emph{Brownian cographon}, 
is constructed from a Brownian excursion. 
By analogy with the realm of permutations~\cite{Nous2, SubsClosedRandomTrees},
 we expect that the Brownian cographon 
(or a one-parameter deformation of it) is a universal limiting object for 
uniform random graphs in classes of graphs 
which are small\footnote{A class of labeled (resp. unlabeled) graphs is \emph{small} when its exponential (resp. ordinary) generating series has positive radius of convergence.} 
and closed under the substitution operation at the core of the modular decomposition. \medskip

\subsection{Main results}~

From now on, for every $n\geq 1$, we let $\bm G_n$ and $\bm G_n^u$ be uniform random labeled and unlabeled cographs of size 
$n$, respectively. 
It is classical (see also \cref{defi:Graphe->Graphon} below) to associate with any graph a graphon, 
and we denote by $W_{\bm G_n}$ and $W_{\bm G_n^u}$ the graphons associated with $\bm G_n$ and $\bm G_n^u$. 

We note that the graphons associated with a labeled graph and its unlabeled version are the same. 
However, $W_{\bm G_n}$ and $W_{\bm G_n^u}$ have different distributions, 
since the number of possible labelings of an unlabeled cograph of a given size varies (see \cref{fig:all_cographs_4} p.\pageref{fig:all_cographs_4} for an illustration). 

\begin{theorem}
\label{th:MainTheorem}
We have the following convergences in distribution as $n$ tends to $+\infty$:
$$
W_{\bm G_n} \to W^{1/2}, \qquad  W_{\bm G_n^u} \to W^{1/2},
$$ 
where $W^{1/2}$ is the Brownian cographon introduced below in \cref{def:BrownianCographon}. 
\end{theorem}
Roughly speaking, the graphon convergence
is the convergence of the rescaled adjacency matrix with an unusual metric, the {\em cut metric}, see \cref{subsec:space_graphons}.
To illustrate \cref{th:MainTheorem}, we show on \cref{fig:simu} the adjacency matrix 
of a large random uniform labeled cograph.  
Entries $1$ in the matrix are represented as black dots, entries $0$ as white dots.
It was obtained by using the encoding of cographs by cotrees and sampling a large uniform cotree using Boltzmann sampling \cite{Boltzmann} of the equation \eqref{eq:SerieS}, p. \pageref{eq:SerieS}.
Note that the order of vertices in the axis of \cref{fig:simu} is not the order of labels but is given by the depth-first search of the associated cotree.
The fractal aspect of the figure -- appearance of imbricated squares at various scale --
is consistent with the convergence to a Brownian limiting object, since the Brownian excursion
enjoys some self-similarity properties.

\begin{figure}[htbp]
\[\includegraphics[height=5cm]{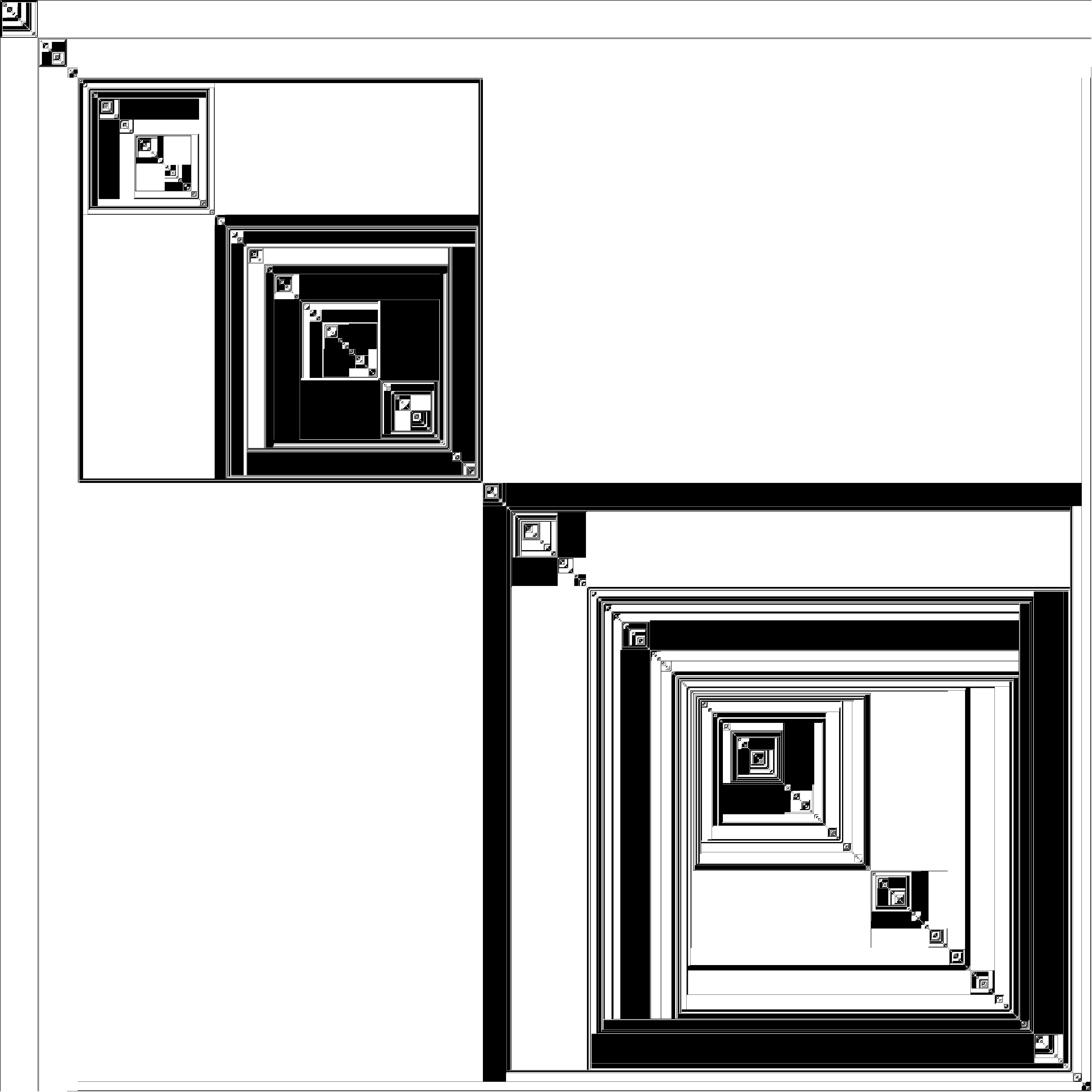}\]
\caption{The adjacency matrix of a uniform labeled random cograph of size 4482. 
\label{fig:simu}}
\end{figure}
\medskip

We now present further results. It is well-known that
the graphon convergence entails the convergence of many graph statistics, 
like subgraph densities, the spectrum of the adjacency matrix, the normalized degree distribution (see \cite{LovaszBook}
and \cref{sec:Graphons} below). 
Hence, \cref{th:MainTheorem} implies that these statistics have the same limit in the labeled and unlabeled cases, 
and that this limit may (at least in principle) be described in terms of the Brownian cographon. 
Among these, the degree distribution of the Brownian cographon (or to be precise, its \emph{intensity}\footnote{The degree distribution of a graphon is a measure, and therefore 
that of the Brownian cographon is a \emph{random} measure.
Following Kallenberg \cite[Chapter 2]{RandomMeasures},
we call the \emph{intensity} of a random measure $\bm \mu$ the (deterministic) measure $I[\bm \mu]$ defined by $I[\bm \mu](A) = \esper[\bm \mu(A)]$ for all measurable sets $A$.
In other words, we consider here the ``averaged'' degree distribution of the Brownian cographon, where we average on all realizations of the Brownian cographon.}) 
is surprisingly nice: it is simply the Lebesgue measure on $[0,1]$. 
We therefore have the following result, where 
we denote by $\deg_G(v)$ the degree of a vertex $v$ in a graph $G$.
\begin{theorem}
\label{th:DegreeRandomVertex}
For every $n\geq 1$, let $\bm v$ and $\bm v^u$ be uniform random vertices in $\bm G_n$ and $\bm G_n^u$, respectively.
We have the following convergences in distribution  as $n$ tends to $+\infty$:
\[\tfrac1n \deg_{\bm G_n}(\bm v) \to U, \qquad \tfrac1n \deg_{\bm G_n^u}(\bm v^u) \to U,\]
where $U$ is a uniform random variable in $[0,1]$.
\end{theorem}

\medskip

On the other hand, other graph statistics are not continuous for the graphon topology, 
and therefore can have different limits in the labeled and unlabeled cases. 
We illustrate this phenomenon with the \emph{vertex connectivity} $\kappa$ (defined as the minimum number of vertices whose removal disconnects the graph). 
Our third result is the following. 

\begin{theorem}\label{thm:DegreeConnectivityIntro}
There exist different probability distributions $(\pi_\formerell)_{\formerell\geq 1}$ and $(\pi_\formerell^u)_{\formerell\geq 1}$ 
such that, for every fixed $\formerell \ge 1$, as $n$ tends to $+\infty$, we have
\begin{equation}
\mathbb{P}(\kappa(\bm G_n)=\formerell) \to\pi_\formerell,
\qquad \mathbb{P}(\kappa(\bm G^u_n)=\formerell) \to \pi^u_\formerell.
\end{equation}
\end{theorem}
Formulas for $\pi_\formerell$ and $\pi_\formerell^u$ are given in \cref{thm:DegreeConnectivity}. 

\begin{remark}
A part of these results (\cref{th:MainTheorem}) has been independently derived in \cite{Benedikt} during the preparation of this paper. 
The proof method is however different. 
\end{remark}

\subsection{Proof strategy}

We first discuss the proof of \cref{th:MainTheorem}. 
For any graphs $g$ and $G$ of size $k$ and $n$ respectively, we denote by $\density(g,G)$ 
the number of copies of $g$ in $G$ as induced subgraphs normalized by $n^k$. 
Equivalently, let $\vec{V}^k$ be a $k$-tuple of i.i.d. uniform random vertices in $G$, then
$\density(g,G)= \mathbb{P}(\SubGraph(\vec{V}^k,G)=g)$, 
where $\SubGraph(I,G)$ is the subgraph of $G$ induced by the vertices of $I$. 
(All subgraphs in this article are induced subgraphs, and we sometimes omit the word ``induced''.) %

From a theorem of Diaconis and Janson~\cite[Theorem 3.1]{DiaconisJanson}, 
the graphon convergence of any sequence of random graphs $({\bm H}_n)$ is characterized 
by the convergence of $\esper[\density(g,{\bm H}_n)]$ for all graphs $g$. 
In the case of $\bm G_n$ (the uniform random labeled cographs of size $n$), for any graph $g$ of size $k$, we have 
\[
 \esper[\density(g,{\bm G}_n)] = \frac{\left|\left\{ (G,I) : \, {G=(V,E) \text{ labeled cograph of size } n, \atop I \in V^k \text{ and } \SubGraph(I,G)=g} \right\}\right|}
 {|\{ G \text{ labeled cograph of size } n \}| \cdot n^k}, 
\]
and a similar formula holds in the unlabeled case. 

Both in the labeled and unlabeled cases, the asymptotic behavior of the denominator 
follows from the encoding of cographs as cotrees, standard application of
the symbolic method of Flajolet and Sedgewick~\cite{Violet} and singularity analysis
(see \cref{prop:Asympt_S_Seven,prop:Asympt_U}). 
The same methods can be used to determine the asymptotic behavior of the numerator, 
counting cotrees with marked leaves inducing a given subtree. 
This requires more involved combinatorial decompositions, which are performed
in \cref{sec:proofLabeled,sec:unlabeled}.

We note that we already used a similar proof strategy in the framework of permutations in~\cite{Nous2}. 
The adaptation to the case of {\em labeled} cographs does not present major difficulties. 
The {\em unlabeled} case is however more subtle, since we have to take care of symmetries when marking leaves in cotrees (see the discussion in \cref{ssec:unlabeled_to_labeled} for details). 
We overcome this difficulty using the $n!$-to-$1$ mapping that maps 
a pair $(G,a)$ (where $G$ is a labeled cograph and $a$ an automorphism of $G$) to the unlabeled version of $G$. 
We then make combinatorial decompositions of such pairs $(G,a)$
with marked vertices inducing a given subgraph
(or more precisely, of the associated cotrees, with marked leaves inducing a given subtree).
Our analysis shows that symmetries have a negligeable influence
 on the asymptotic behavior of the counting series. 
This is similar -- though we have a different and more combinatorial presentation -- to the techniques
 developed in the papers~\cite{PanaStufler,GittenbergerPolya},
devoted to the convergence of unordered unlabeled trees to the \emph{Brownian Continuum Random Tree}. 
\medskip

With \cref{th:MainTheorem} in our hands, proving 
\cref{th:DegreeRandomVertex} amounts to proving that 
the intensity of the degree distribution of the Brownian cographon is the Lebesgue measure on $[0,1]$.
Rather than working in the continuous,
we exhibit a discrete approximation $\bm G_n^b$ of the Brownian cographon,
which has the remarkable property that the degree of a uniform random vertex 
in $\bm G_n^b$ is {\em exactly} distributed as a uniform random variable in $\{0,1,\cdots,n-1\}$.
The latter is proved by purely combinatorial arguments (see \cref{prop:IntensityW12}). 

\medskip

To prove \cref{thm:DegreeConnectivityIntro}, we start with a simple combinatorial lemma, 
which relates the vertex connectivity of a connected cograph to the sizes 
of the subtrees attached to the root in its cotree. 
Based on that, we can use again the symbolic method and singularity analysis as in the proof of \cref{th:MainTheorem}. 

\subsection{Outline of the paper}
\cref{sec:cotrees} explains the standard encoding of cographs by cotrees
and the relation between taking induced subgraphs and subtrees.
\cref{sec:Graphons} presents the necessary material on graphons;
results stated there are quoted from the literature, except the continuity of the degree distribution,
for which we could not find a reference.
\cref{sec:BrownianCographon} introduces the limit object of \cref{th:MainTheorem},
namely the Brownian cographon. It is also proved that the intensity of its degree distribution is uniform
(which is the key ingredient for \cref{th:DegreeRandomVertex}).
\cref{th:MainTheorem,th:DegreeRandomVertex} are proved in \cref{sec:proofLabeled} 
for the labeled case and in \cref{sec:unlabeled} for the unlabeled case.
Finally, \cref{thm:DegreeConnectivityIntro} is proved in \cref{sec:DegreeConnectivity}.

\section{Cographs, cotrees and induced subgraphs}
\label{sec:cotrees}
\subsection{Terminology and notation for graphs}
All graphs considered in this paper are \emph{simple} (\emph{i.e.} without multiple edges, nor loops) and not directed.
A {\em labeled graph} $G$ is a pair $(V,E)$,
where $V$ is its vertex set (consisting of distinguishable vertices, each identified by its label) 
and $E$ is its edge set.
Two labeled graphs $(V,E)$ and $(V',E')$ are isomorphic if there exists a bijection
from $V$ to $V'$ which maps $E$ to $E'$.
Equivalence classes of labeled graphs for the above relation are {\em unlabeled graphs}.

Throughout this paper, the \emph{size} of a graph is its number of vertices.
Note that there are finitely many unlabeled graphs with $n$ vertices,
so that the uniform random unlabeled graph of size $n$ is well defined. 
For labeled graphs, there are finitely many graphs with any given vertex set $V$.
Hence, to consider a uniform random labeled graph of size $n$, 
we need to fix a vertex set $V$ of size $n$.
The properties we are interested in do not depend on the choice of this vertex set, so that we can choose $V$ arbitrarily,
usually $V=\{1,\dots,n\}$.

As a consequence, considering a subset (say $\mathcal{C}$) of the set of all graphs, 
we can similarly define the \emph{uniform random unlabeled graph of size $n$ in $\mathcal{C}$} 
(resp. the uniform random labeled graph with vertex set $\{1,\dots,n\}$ in $\mathcal{C}$ -- 
which we simply denote by \emph{uniform random labeled graph of size $n$ in $\mathcal{C}$}). 
The restricted family of graphs considered in this paper is that of \emph{cographs}.

\subsection{Cographs and cotrees}
Let $G=(V,E)$ and $G'=(V',E')$ be labeled graphs with disjoint vertex sets.
We define their \emph{disjoint union} as the graph $(V \uplus V', E \uplus E')$ (the symbol $\uplus$ denoting as usual the disjoint union of two sets). 
We also define their \emph{join} as the graph
 $(V \uplus V', E \uplus E' \uplus (V \times V'))$: 
 namely, we take copies of $G$ and $G'$, 
 and add all edges from a vertex of $G$ to a vertex of $G'$.
 Both definitions readily extend to more than two graphs 
 (adding edges between any two vertices originating from different graphs in the case of the join operation).

\begin{definition}
A \emph{labeled cograph} is a labeled graph that can be generated from single-vertex graphs 
 applying join and disjoint union operations. An \emph{unlabeled cograph} is the underlying unlabeled graph of a labeled cograph.
\end{definition}

It is classical to encode cographs by their \emph{cotrees}. 

\begin{definition}
  \label{def:cotree}
A labeled \emph{cotree} of size $n$ is a rooted tree $t$ with $n$ leaves labeled from $1$ to $n$ such that
\begin{itemize}
\item $t$ is not plane (\emph{i.e.} the children of every internal node are not ordered);%
\item every internal node has at least two children;
\item every internal node in $t$ is decorated with a $\Zero$ or a $\One$.
\end{itemize}
An \emph{unlabeled cotree} of size $n$ is a labeled cotree of size $n$ where we forget the labels on the leaves.
\end{definition}

\begin{remark}
In the literature, cotrees are usually required to satisfy the property that decorations $\Zero$ and $\One$ 
should alternate along each branch from the root to a leaf. 
In several proofs, our work needs also to consider trees in which this alternation assumption is relaxed, 
hence the choice of diverging from the usual terminology. 
Cotrees which do satisfy this alternation property are denoted \emph{canonical cotrees} in this paper (see \cref{def:canonical_cotree}). 
\end{remark}

For an unlabeled cotree $t$, we denote by $\cograph(t)$ the unlabeled graph defined recursively as follows (see an illustration in \cref{fig:ex_cotree}):
\begin{itemize}
\item If $t$ consists of a single leaf, then $\cograph(t)$ is the graph with a single vertex.
\item Otherwise, the root of $t$ has decoration $\Zero$ or $\One$ and has subtrees $t_1$, \dots, $t_d$
attached to it ($d \ge 2$).
Then, if the root has decoration $\Zero$, we let $\cograph(t)$ be the {\em disjoint union}
of $\cograph(t_1)$, \dots, $\cograph(t_d)$. 
Otherwise, when the root has decoration $\One$,
 we let $\cograph(t)$ be the {\em join}
of $\cograph(t_1)$, \dots, $\cograph(t_d)$.
\end{itemize}
Note that the above construction naturally entails a one-to-one correspondence 
between the leaves of the cotree $t$ and the vertices of its associated graph $\cograph(t)$. 
Therefore, it maps the size of a cotree to the size of the associated graph. 
Another consequence is that we can extend the above construction to a \emph{labeled} cotree $t$, 
and obtain a \emph{labeled} graph (also denoted $\cograph(t)$), with vertex set $\{1,\dots,n\}$: 
each vertex of $\cograph(t)$ receives the label of the corresponding leaf of $t$. 

By construction, for all cotrees $t$, the graph $\cograph(t)$ is a cograph.
Conversely, each cograph can be obtained in this way, albeit not from a unique tree $t$.
It is however possible to find a canonical cotree representing a cograph $G$. 
This tree was first described in~\cite{corneil}. 
The presentation of~\cite{corneil}, although equivalent, is however a little bit different, 
since cographs are generated using exclusively ``complemented unions''
instead of disjoint unions and joins. 
The presentation we adopt has since been used in many algorithmic papers, see \emph{e.g.}~\cite{Habib,Bretscher}. 

\begin{definition}\label{def:canonical_cotree}
A cotree is \emph{canonical} 
if every child of a node decorated by $\Zero$ (resp. $\One$) is either decorated by $\One$ (resp. $\Zero$) or a leaf.
\end{definition}

\begin{proposition}\label{prop:canonical_cotree}
Let $G$ be a labeled (resp. unlabeled) cograph. 
Then there exists a unique labeled (resp. unlabeled) canonical cotree $t$ such that $\cograph(t)=G$. 
\end{proposition}

Example of cographs and their canonical cotree are given in \cref{fig:ex_cotree,fig:all_cographs_4}.

From a graph $G$, the canonical cotree $t$ such that $\cograph(t)=G$ is recursively built as follows. 
If $G$ consists of a single vertex, $t$ is the unique cotree with a single leaf. 
If $G$ has at least two vertices, we distinguish cases depending on whether $G$ is connected or not. 
\begin{itemize}
 \item If $G$ is not connected, the root of $t$ is decorated with $\Zero$ and the subtrees attached to it are the cographs associated with the connected components of $G$. 
 \item If $G$ is connected, the root of $t$ is decorated with $\One$ and the subtrees attached to it are the cographs associated with the induced subgraphs of $G$ 
whose vertex sets are those of the connected components of $\bar{G}$, where $\bar{G}$ is the complement of $G$ 
(graph on the same vertices with complement edge set). 
\end{itemize}
Important properties of cographs which justify the correctness of the above construction are the following: 
cographs are stable by induced subgraph and by complement, and a cograph $G$ of size at least two is not connected exactly when its complement $\bar{G}$ is connected. 

\begin{figure}[htbp]
\begin{center}
\includegraphics[width=8cm]{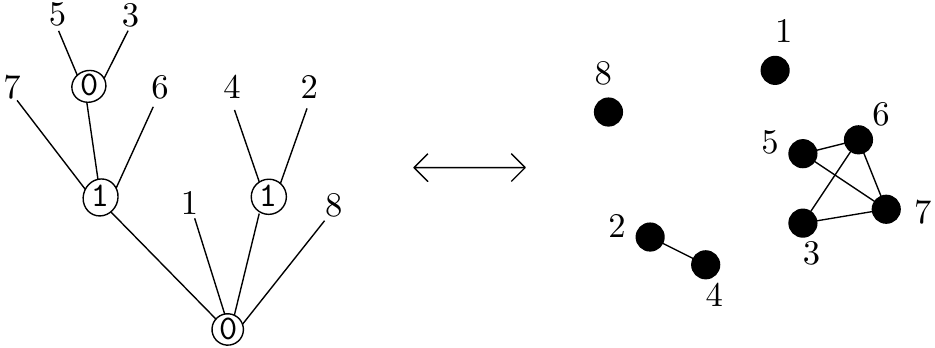}
\caption{Left: A labeled canonical cotree $t$ with $8$ leaves. Right: The associated labeled cograph $\cograph(t)$ of size $8$. }
\label{fig:ex_cotree}
\end{center}
\end{figure}

 \begin{figure}[htbp]
\begin{center}
\begin{tabular}[t]{| >{\centering\arraybackslash}m{40mm}  >{\centering\arraybackslash}m{20mm} >{\centering\arraybackslash}m{20mm}>{\centering\arraybackslash}m{20mm}>{\centering\arraybackslash}m{20mm}>{\centering\arraybackslash}m{20mm} |}
\hline
 & & & & & \\
Unlabeled canonical cotree   & \includegraphics[width=18mm]{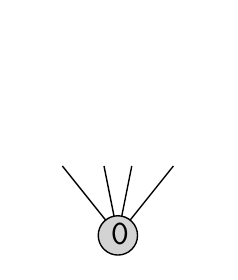} & \includegraphics[width=18mm]{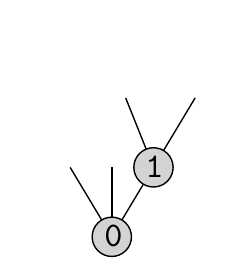} & \includegraphics[width=18mm]{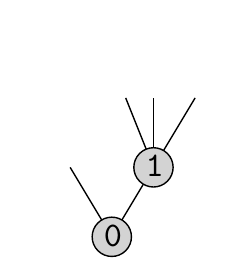} & \includegraphics[width=18mm]{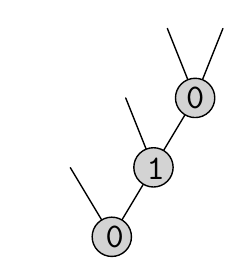} & \includegraphics[width=18mm]{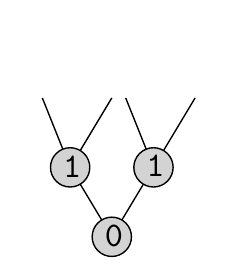}\\

Corresponding unlabeled cograph   & \includegraphics[width=18mm]{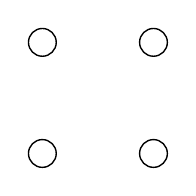} & \includegraphics[width=18mm]{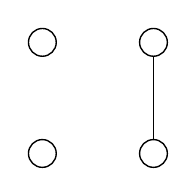} & \includegraphics[width=18mm]{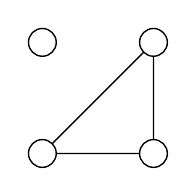} & \includegraphics[width=18mm]{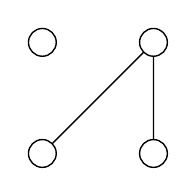} & \includegraphics[width=18mm]{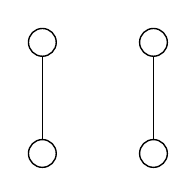}\\
Number of associated labeled cographs & $1$ & $6$ & $4$ & $12$ & $3$  \\
 & & & & & \\
\hline
 & & & & & \\
Unlabeled canonical cotree   & \includegraphics[width=18mm]{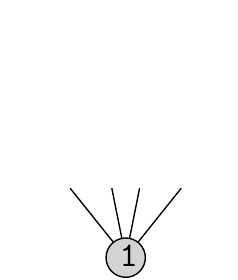} & \includegraphics[width=18mm]{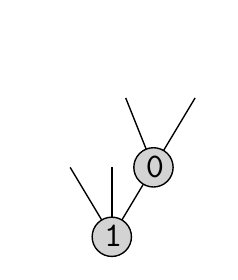} & \includegraphics[width=18mm]{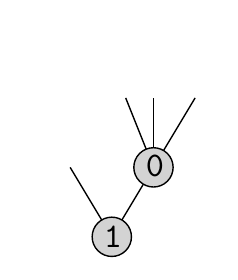} & \includegraphics[width=18mm]{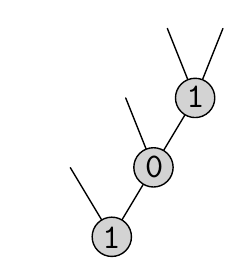} & \includegraphics[width=18mm]{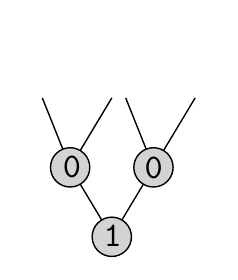}\\
Corresponding unlabeled cograph    & \includegraphics[width=18mm]{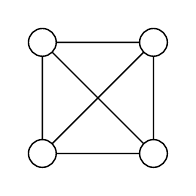} & \includegraphics[width=18mm]{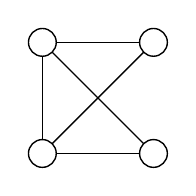} & \includegraphics[width=18mm]{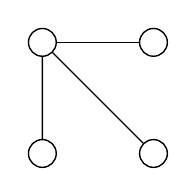} & \includegraphics[width=18mm]{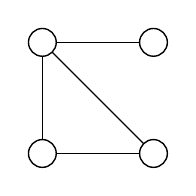} & \includegraphics[width=18mm]{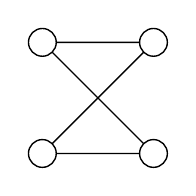}\\
Number of associated labeled cographs & $1$ & $6$ & $4$ & $12$ & $3$  \\
 & & & & & \\
\hline
\end{tabular}
\caption{All unlabeled cographs of size $4$ with their corresponding (unlabeled) canonical cotrees
and their number of distinct labelings.}
 \label{fig:all_cographs_4}
\end{center}
\end{figure}

\subsection{Subgraphs and induced trees}

Let $G$ be a graph of size $n$ (which may or not be labeled), and let $I=(v_1,\dots,v_k)$ be a $k$-tuple of vertices of $G$.
Recall that the \emph{subgraph of $G$ induced by $I$},
which we denote by $\SubGraph(I,G)$, is the graph with vertex set $\{v_1,\dots,v_k\}$ and which contains the edge $\{v_i,v_j\}$ 
if and only if $\{v_i,v_j\}$ is an edge of $G$. 
In case of repetitions of vertices in $I$, we take as many copies of each vertex as times it appears in $I$
and do not connect copies of the same vertex.
We always regard $\SubGraph(I,G)$ as an unlabeled graph.

In the case of cographs, the (induced) subgraph operation can also be realized on the cotrees,
through {\em induced trees}, which we now present.
We start with a preliminary definition.
\begin{definition}[First common ancestor]\label{dfn:common_ancestor}
Let $t$ be a rooted tree, and $u$ and $v$ be two nodes (internal nodes or leaves) of $t$. 
The \emph{first common ancestor} of $u$ and $v$ is the node furthest away from the root $\varnothing$ that appears 
on both paths from $\varnothing$ to $u$ and from $\varnothing$ to $v$ in $t$. 
\end{definition}

For any cograph $G$, and any vertices $i$ and $j$ of $G$, %
the following simple observation allows to read in any cotree encoding $G$ if $\{i,j\}$ is an edge of $G$.

\begin{observation}\label{obs:caract_edge-label1}
Let $i \neq j$ be two leaves of a cotree $t$ and $G=\cograph(t)$. 
We also denote by $i$ and $j$ the corresponding vertices in $G$. 
Let $v$ be the first common ancestor of $i$ and $j$ in $t$.
Then $\{i,j\}$ is an edge of $G$ if and only if $v$ has label $\One$ in $t$.
\end{observation}

\begin{definition}[Induced cotree]\label{dfn:induced_cotree}
Let $t$ be a cotree (which may or not be labeled), and $I=(\ell_1,\dots,\ell_k)$ a $k$-tuple of distinct leaves of $t$, 
which we call the \emph{marked leaves} of $t$.
The \emph{tree induced by $(t,I)$}, denoted $t_I$, is the \textbf{always labeled} cotree of size $k$ defined as follows.
The tree structure of $t_I$ is given by
\begin{itemize}
\item the leaves of $t_I$ are the marked leaves of $t$;%
 \item the internal nodes of $t_I$ are the nodes of $t$ that are first common ancestors of two (or more) marked leaves; 
 \item the ancestor-descendant relation in $t_I$ is inherited from the one in $t$; 
 \item the decoration of an internal node $v$ of $t_I$ is inherited from the one in $t$; 
\item for each $i\leq k$, the leaf of $t_I$ corresponding to leaf $\ell_i$ in $t$ is labeled $i$ in $t_I$.
\end{itemize}
\end{definition}

We insist on the fact that we \textbf{always} define the induced cotree $t_I$ as a \textbf{labeled} cotree,
regardless of whether the original cotree $t$ is labeled or not.
The labeling of the induced cotree is related to the order of the marked leaves $I$ 
(and not to their labels in the case $t$ was labeled); 
this echoes the choice of $I$ as $k$-tuple (\emph{i.e.} an {\em ordered} collection) of distinct leaves.
A detailed example of the induced cotree construction is given in \cref{fig:SsArbreInduit}.

\begin{figure}[htbp]
    \begin{center}
      \includegraphics[width=12cm]{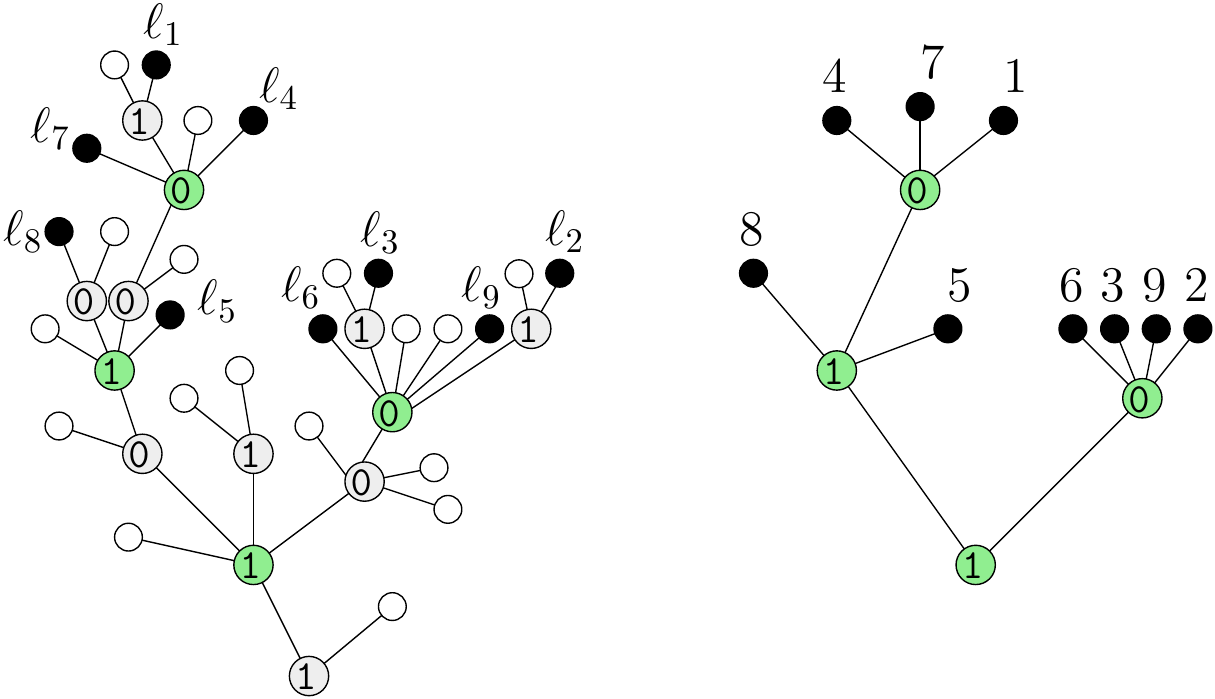}
    \end{center}
\caption{On the left: A cotree $t$ of size $n=26$, where leaves are indicated both by $\circ$ and $\bullet$. 
We also fix a $9$-tuple $I=(\ell_1,\dots,\ell_9)$ of marked leaves (indicated by $\bullet$). 
In green, we indicate the internal nodes of $t$ which are first common ancestors of these $9$ marked leaves.
On the right: The labeled cotree $t_I$ induced by the $9$ marked leaves.}
    \label{fig:SsArbreInduit}
\end{figure}
\begin{proposition}\label{prop:diagramme_commutatif}
Let $t$ be a cotree and $G=\cograph(t)$ the associated cograph.
Let $I$ be a $k$-tuple of distinct leaves in $t$, which identifies a $k$-tuple 
of distinct vertices in $G$.
Then, as unlabeled graphs, we have $\SubGraph(I,G) = \cograph(t_I)$.
\end{proposition}

\begin{proof}
This follows immediately from \cref{obs:caract_edge-label1} and the fact that the induced cotree construction (\cref{dfn:induced_cotree}) 
preserves first common ancestors and their decorations. 
\end{proof}

\begin{remark}
  The reader might be surprised by the fact that
  we choose induced subtrees to be labeled,
  while induced subgraphs are not.
  The reasons are the following.
  The labeling of induced subtrees avoids symmetry problems
  when decomposing cotrees with marked leaves inducing a given subtree,
  namely in \cref{th:SerieT_t0,th:V_t_0}.
  On the other hand, the theory of graphons 
  is suited to consider unlabeled subgraphs.
\end{remark}

%
%
%
%
%
%
%
%
%
%
%
%
%
%
%
%
%
%

\section{Graphons}
\label{sec:Graphons} 
Graphons are continuum limit objects for sequences of graphs.
We present here the theory relevant to our work.
We recall basic notions from the literature,
following mainly Lovász' book \cite{LovaszBook},
then we recall results of Diaconis and Janson \cite{DiaconisJanson} 
regarding convergence of random graphs to random graphons. 
Finally, we prove a continuity result for the degree distribution with respect to graphon convergence.

The theory of graphons classically deals with unlabeled graphs.
Consequently,
unless specified otherwise, all graphs in this section are considered unlabeled.
When considering labeled graphs, graphon convergence is to be understood as
the convergence of their unlabeled versions.

\subsection{The space of graphons}\label{subsec:space_graphons}
\begin{definition}\label{defi:Graphon}
A graphon is an equivalence class of symmetric functions $[0,1]^2 \to [0,1]$,
under the equivalence relation $\sim$, where $w \sim u$ if 
there exists an invertible  measurable and Lebesgue measure preserving function $\phi:[0,1] \to [0,1]$
such that  $w(\phi(x),\phi(y)) = u(x,y)$ for almost every $x,y\in[0,1]$.
\end{definition}
Intuitively, a graphon is a continuous analogue of the adjacency matrix of a graph,
viewed up to relabelings of its continuous vertex set.
\begin{definition}\label{defi:Graphe->Graphon}
	The graphon $W_G$ associated to a labeled graph $G$ with $n$ vertices (labeled from $1$ to $n$) is the equivalence class of the function $w_G:[0,1]^2\to [0,1]$ where
	\[w_G(x,y) = A_{\lceil nx\rceil,\lceil ny\rceil} \in \{0,1\}\]
and $A$ is the adjacency matrix of the graph $G$.
\end{definition}
Since any relabeling of the vertex set of $G$ gives the same graphon $W_G$, 
the above definition immediately extends to unlabeled graphs. 

We now define the so-called {\em cut metric}, first on functions, and then on graphons.
We note that it is different than usual metrics on spaces
of functions ($L^1$, supremum norms, \ldots),
see \cite[Chapter 8]{LovaszBook} for details. 
For a real-valued symmetric function $w$ on $[0,1]^2$, its {\em cut norm} is defined as
\[ \| w \|_\Box = \sup_{S,T \subseteq [0,1]} \left| \int_{S \times T} w(x,y) dx dy \right| \]
Identifying as usual functions equal almost-everywhere, this is indeed a norm.
It induces the following {\em cut distance} on the space of graphons
\[ \dbox(W,W') = \inf_{w \in W,w' \in W'} \| w -w'\|_\Box.\]
While the symmetry and triangular inequalities are immediate,
this ``distance'' $\dbox$ {\em does not} separate points,
{\em i.e.} there exist different graphons at distance zero.
Call $\SpaceGraphon$ the space of graphons, quotiented by the equivalence
relation $W \equiv W'$ if $\dbox(W,W') =0$.
This is a metric space with distance $\dbox$.
This definition is justified by the following deep result, see,
  {\em e.g.}, \cite[Theorem 9.23]{LovaszBook}.
\begin{theorem}
  \label{thm:Graphons_Compact}
  The metric space $(\SpaceGraphon,\dbox)$ is compact.
\end{theorem}
In the sequel, we think of graphons as elements in $\SpaceGraphon$ and
convergences of graphons are to be understood with respect to the distance $\dbox$.

\medskip

\subsection{Subgraph densities and samples}
~

An important feature of graphons is that one can 
extend the notion of density of a given subgraph $g$ in a graph $G$
to density in a graphon $W$.
Moreover, convergence of graphons turns to be equivalent to convergence of all subgraph densities.

To present this, we start by recalling from the introduction the definition of 
subgraph densities in graphs.
We recall that, if $I$ is a tuple of vertices of $G$,
then we write $\SubGraph(I,G)$ for the induced subgraph of $G$
on vertex set $I$. %

\begin{definition}[Density of subgraphs]
The density of an (unlabeled) graph $g$ of size $k$ in a graph $G$ of size $n$ (which may or not be labeled) is defined as follows:
let $\vec{V}^k$ be a $k$-tuple of i.i.d. uniform random vertices in $G$, then
\begin{align*}
\density(g,G) = \mathbb{P}(\SubGraph(\vec{V}^k,G)=g).
\end{align*}
\end{definition}

We now extend this to graphons.
Consider a graphon $W$ and one of its representatives $w$.
We denote by $\Sample_k(W)$ the unlabeled random graph built as follows: 
$\Sample_k(W)$ has vertex set $\{v_1,v_2,\dots,v_k\}$ and,
letting $\vec{X}^k=(X_1,\dots,X_k)$ be i.i.d. uniform random variables in $[0,1]$,
we connect vertices $v_i$ and $v_j$ with probability $w(X_i,X_j)$
(these events being independent, conditionally on $(X_1,\cdots,X_k)$).
Since the $X_i$'s are independent and uniform in $[0,1]$, 
the distribution of this random graph is the same
if we replace $w$ by a function $w' \sim w$ in the sense of 
\cref{defi:Graphon}.
It turns out that this distribution also stays the same if we replace $w$
by a function $w'$ such that $\|w-w'\|_\Box=0$ (it can be seen as a consequence of \cref{Thm:GraphonsCv_Densities} below),
so that the construction is well-defined on graphons.

\begin{definition} %
The density of a graph $g=(V,E)$ of size $k$ in a graphon $W$ is 
\begin{align*}
\density(g,W) &= \mathbb{P}(\Sample_k(W)=g)\\
&= \int_{[0,1]^k} \prod_{\{v,v'\} \in E} w(x_v,x_{v'}) \prod_{\{v,v'\} \notin E} (1-w(x_v,x_{v'})) \, \prod_{v \in V} dx_v,
\end{align*}
where, in the second expression, we choose an arbitrary representative $w$ in $W$.
\end{definition}
This definition extends that of the density of subgraphs in (finite) graphs in the following sense. For every finite graphs $g$ and $G$,
denoting by $\vec{V}^k$ a $k$-tuple of i.i.d. uniform random vertices in $G$,
$$
\density(g,W_G) = \mathbb{P}(\Sample_k(W_G)=g) = \mathbb{P}(\SubGraph(\vec{V}^k,G)=g) = \density(g,G).
$$
The following theorem is a prominent result in the theory of graphons,
see {\em e.g.} \cite[Theorem 11.5]{LovaszBook}.
\begin{theorem}
  \label{Thm:GraphonsCv_Densities}
  Let $W_n$ (for all $n \geq 0$) and $W$ be graphons. Then the following are equivalent:
	\begin{enumerate}%
		\item [(a)] $(W_n)$ converges to $W$ (for the distance $\dbox$);
        \item [(b)] for any fixed graph $g$, we have $\density(g,W_n) \to \density(g,W)$.
    \end{enumerate}
\end{theorem}

Classically, when $(H_n)$ is a sequence of graphs, we say that $(H_n)$ converges to a graphon $W$ when $(W_{H_n})$ converges to $W$. 

\medskip

\subsection{Random graphons}
~

We now discuss convergence of a sequence of random graphs $\bm H_n$ (equivalently, of the associated random graphons $W_{\bm H_n}$)
towards a possibly random limiting graphon $\bm W$.
In this context, the densities $\density(g,\bm H_n)$ are random variables.
This was studied in \cite{DiaconisJanson} and it turns out that 
it is enough to consider the expectations $\esper\big[\density(g,\bm H_n)\big]$ and $\esper\big[\density(g,\bm W)\big]$ 
to extend \cref{Thm:GraphonsCv_Densities} to this random setting. 
Note first that $\esper\big[\density(g,\bm H_n)\big] = \proba(\SubGraph(\vec{V}^k,\bm H_n) = g)$, 
where both $\bm H_n$ and $\vec{V}^k$ are random, 
and that similarly $\esper\big[\density(g,\bm W)\big]= \proba(\Sample_k(\bm W)=g)$,
where the randomness comes both from $\bm W$ and the operation $\Sample_k$.

A first result states that the distributions of random graphons are characterized by expected subgraph densities.
\begin{proposition}[Corollary 3.3 of \cite{DiaconisJanson}]\label{Prop:CaracterisationLoiGraphon}
Let ${\bf W},{\bf W'}$ be two random graphons, seen as random variables in $\SpaceGraphon$.
The following are equivalent: 
\begin{itemize}
 \item ${\bf W}\stackrel{d}{=}{\bf W'}$;
 \item for every finite graph $g$, $\mathbb{E}[\density(g,{\bf W})] = \mathbb{E}[\density(g,{\bf W'})]$;
 \item for every $k\geq 1$, $\Sample_k(\bm W) \stackrel d = \Sample_k(\bm W')$.
\end{itemize}
\end{proposition}

The next result, which is essentially \cite[Theorem 3.1]{DiaconisJanson},
characterizes the convergence in distribution of random graphs to random graphons.
\begin{theorem}
\label{th:Portmanteau}
	For any $n$, let $\bm H_n$ be a random graph of size $n$. Denote by $W_{\bm H_n}$ the graphon associated to ${\bm H_n}$ by \cref{defi:Graphe->Graphon}.
	The following assertions are equivalent:
	\begin{enumerate}%
		\item [(a)] The sequence of random graphons $(W_{\bm H_n})_n$ converges in distribution to some random graphon $\bm W$. 
		\item [(b)] The random infinite vector $\big(\density(g,\bm H_n)\big)_{g\text{ finite graph}}$ converges in distribution in the product topology to some random infinite vector $(\bm \Lambda_g)_{g\text{ finite graph}}$. 
		\item [(c)]For every finite graph $g$, there is a constant $\Delta_g \in [0,1]$ such that \[\mathbb{E}[\density(g,\bm H_n)] \xrightarrow{n\to\infty} \Delta_g.\]

        \item [(d)] For every $k\geq 1$, denote by $\vec{V'}^k=(V'_1,\dots,V'_k)$ a uniform $k$-tuple of {\em distinct} vertices of $\bm H_n$.
Then the induced graph $\SubGraph(\vec{V'}^k,\bm H_n)$ converges in distribution to some random graph $\bm g_k$.
	\end{enumerate}
Whenever these assertions are verified, we have 
\begin{equation}
  \label{eq:Lambda_Density}
(\bm \Lambda_g)_{g\text{ finite graphs}} \stackrel d = (\density(g,\bm W))_{g\text{ finite graphs}}.
\end{equation}
and, for every graph $g$ of size $k$,
\begin{equation}
  \Delta_g = \esper[\bm \Lambda_g] = \esper[\density(g,\bm W)]= \proba(\bm g_k = g). 
  \label{eq:Delta_Esper}
\end{equation}
\end{theorem}
Using the identity $\esper\big[\density(g,\bm W)\big]= \proba(\Sample_k(\bm W)=g)$,
we note that \cref{eq:Delta_Esper} implies that, for all $k \ge 1$, we have
\begin{equation}
  \Sample_k(\bm W) \stackrel d = \bm g_k 
  \label{eq:SampleK_gK}
\end{equation}

\begin{proof}
  The equivalence of the first three items, \cref{eq:Lambda_Density}
  and the first two equalities in \cref{eq:Delta_Esper} are all proved in \cite{DiaconisJanson}; see Theorem 3.1 there.
  Thus, we only prove (c) $\Leftrightarrow$ (d) and the related equality
  $\mathbb{P}(\bm g_k = g)=\Delta_g$.
  
  For any graphs $g,G$ of respective sizes $k\leq n$,  we define their \emph{injective density} 
  $ \density^{\text{inj}}(g,G) = \mathbb{P}(\SubGraph(\vec{V'}^k,G)=g)$ where $\vec{V'}^k$ is a uniform $k$-tuple of {\em distinct} vertices of $G$.
  As explained in~\cite{DiaconisJanson} (and standard in the graphon literature),
  Assertion (c) is equivalent, for the same limits $(\Delta_g)$, to its analogue with injective densities, 
  which is: for every graph $g$,
  \begin{equation}  
    \mathbb{E}[\density^{\text{inj}}(g,\bm H_n)] \xrightarrow{n\to\infty} \Delta_g.
    \label{eq:conv_Inj_Dens} 
  \end{equation} 
    Moreover, we note that, if $(\bm H_n)$ is a sequence of random graphs, then, for any graph $g$ of size $k$, 
    \begin{equation}   \label{eq:esper_Dens_Inj}   
    \esper \big[ \density^{\text{inj}}(g,\bm H_n) \big]   = \mathbb{P}(\SubGraph(\vec{V'}^k,\bm H_n)=g),
  \end{equation}
  where both $\vec{V'}^k$ and $\bm H_n$ are random. 
  Since $\SubGraph(\vec{V'}^k,\bm H_n)$ takes value in a finite set, its convergence in distribution (Assertion (d)) is equivalent to the convergence of its point probabilities, {\em i.e.} of the right-hand side of \cref{eq:esper_Dens_Inj}.
  Recalling \cref{eq:conv_Inj_Dens}, this proves the equivalence of Assertions (c) and (d).
  Futhermore, when these assertions hold, we have
  \[\mathbb{P}(\bm g_k = g)= \lim_{n \to \infty} \mathbb{P}\big[ \SubGraph(\vec{V'}^k,\bm H_n)=g \big] = \lim_{n \to \infty} \esper \big[ \density^{\text{inj}}(g,\bm H_n) \big] = \Delta_g,\]
  as wanted.
  \end{proof}
We finally collect an immediate corollary.
\begin{lemma}\label{lem:sample_converges_to_W}
	If $\bm W$ is a random graphon, then $W_{\Sample_n(\bm W)}$ converges in distribution to $\bm W$ as $n\to\infty$.
\end{lemma}

\begin{proof}
 Recall that $\Sample_n(\bm W)$ is the random graph on vertex set $\{v_1,\cdots,v_n\}$
 obtained by taking $X_1$, \ldots, $X_n$ i.i.d. uniform in $[0,1]$ and joining $v_i$ and $v_j$
 with probability $w(X_i,X_j)$ where $w$ is a representative of $W$.
 Fix $k$ in $\{1,\cdots,n\}$. 
As in the previous theorem, let $\vec{V'}^k=(v_{h_1},\cdots,v_{h_k})$
 be a uniform random $k$-tuple of distinct vertices of $\Sample_n(\bm W)$.
 Then $\SubGraph(\vec{V'}^k,\Sample_n(\bm W))$ is the random graph on vertex set $\{v_{h_1},\cdots,v_{h_k}\}$
  obtained by taking $X_{h_1}$, \ldots, $X_{h_k}$ i.i.d. uniform in $[0,1]$ and joining $v_{h_i}$ and $v_{h_j}$
  with probability $w(X_{h_i},X_{h_j})$.
 Up to renaming $v_{h_i}$ as $v_i$ and $X_{h_i}$ as $X_i$, this matches the construction of $\Sample_k(\bm W)$.
  Therefore we have the following equality in distribution of random unlabeled graphs:
	\[\SubGraph(\vec{V'}^k,\Sample_n(\bm W)) \stackrel d= \Sample_k(\bm W).\]
    Thus, Assertion (d) of \cref{th:Portmanteau} is fulfilled for the graph sequence $(\Sample_n(\bm W))_n$
	and for $\bm g_k \stackrel d = \Sample_k(\bm W)$.
	Therefore Assertion (a) holds and the graphon sequence $(W_{\Sample_n(\bm W)})_n$
    has a limit in distribution $\bm W'$, which satisfies, for all $k$ (see \cref{eq:SampleK_gK}):
	\[\Sample_k(\bm W') \stackrel d = \bm g_k \stackrel d = \Sample_k(\bm W). \]
	From \cref{Prop:CaracterisationLoiGraphon}, we have $\bm W  \stackrel d= \bm W'$,
	concluding the proof of the lemma.
\end{proof}

\subsection{Graphons and degree distribution}
In this section, we consider the degree distribution of a graphon,
as introduced by Diaconis, Holmes and Janson in \cite[Section 4]{diaconis2008threshold}.
This defines a continuous functional from the space of graphons to that
of probability measures on $[0,1]$ (equipped with the weak topology)
\cite[Theorem 4.2]{diaconis2008threshold}.
For self-completeness we provide a proof here of this simple fact.
Our approach is different from that of Diaconis, Holmes and Janson:
we prove that the functional is 2-Lipschitz with respect to natural metrics,
while they express moments of the degree distribution
in terms of subgraph densities to prove the continuity.

At the end of the section, we also discuss the degree distribution of random graphs
and random graphons. This is a preparation for \cref{ssec:degdistrib}  
where we shall study the degree distribution of random cographs and of the Brownian cographon. 
We also note that other works \cite{BickelChenLevina,DegreeGraphonsDelmas}
study of the degree distributions of specific random graph models
and their convergence to that of their graphon limit.

The {\em degree distribution} of a graphon $W$ is the measure $D_{W}$ on $[0,1]$
defined as follows (see \cite[Theorem 4.4]{diaconis2008threshold}):
 for every continuous bounded function $f : [0,1] \to \R $,
we have
\[ \int_{[0,1]}f(x) D_{W}(dx) = \int_{[0,1]} f\left( \int_{[0,1]} w(u,v)dv \right) du,\]
where $w$ is, as usual, an arbitrary representative of $W$ 
(the resulting measure does not depend on the chosen representative).

For the graphon $W_G$ associated to a graph $G$ of size $n$,
the measure $D_{W_G}$ is simply the empirical distribution of the rescaled degrees:
\[D_{W_G} = \frac 1 {n}\sum_{v\in G} \delta_{\deg_G(v)/n}\]
where $\delta_u$ is the Dirac measure concentrated at $u$.

The next lemma implies that graphon convergence entails weak convergence of degree distributions.
To state a more precise result,
it is useful to endow the space $\mathcal M_1([0,1])$ of Borel probability measures 
on $[0,1]$ with the so-called Wasserstein metric
(see {\em e.g.} \cite[Section 1.2]{Ross_Survey_Stein}), defined as
\[d_{\mathrm{Wass}}(\mu,\nu)= \inf_f \left|\int_{[0,1]}f(x) \mu(dx) - \int_{[0,1]}f(x) \nu(dx)\right|, \]
where the infimum runs over all $1$-Lipschitz functions $f$ from $[0,1]$ to $\mathbb R$.
We recall that this distance metrizes weak convergence (see \emph{e.g.} \cite[Sec. 8.3]{Bogachev}).
\begin{lemma}\label{lem:cv-degree-distribution}
   The map $W\mapsto D_W$ from $(\SpaceGraphon,\dbox)$ to $(\mathcal M_1([0,1]),d_{\mathrm{Wass}})$ is 2-Lipschitz.
   Consequently, if $(W_n)$ converges to $W$ in $\SpaceGraphon$,
   then the sequence of associated measures $(D_{W_n})$ converges weakly to $D_W$.
\end{lemma}
\begin{proof} 
  Let $W$ and $W'$ be graphons with representatives $w$ and $w'$.
  Let $f:[0,1] \to \mathbb R$ be 1-Lipschitz. We have
  \begin{align*}
    d_{\mathrm{Wass}}(D_W, D_{W'})  & \leq \left|\int_{[0,1]}f(x) D_{W}(dx) - \int_{[0,1]}f(x) D_{W'}(dx)\right|\\
    &= \left|\int_{[0,1]}f\left( \int_{[0,1]} w(u,v)dv \right) - f\left( \int_{[0,1]} w'(u,v)dv \right)du\right| \\
    &\leq \int_{[0,1]} \left| \int_{[0,1]} (w(u,v)-  w'(u,v))dv\right|du\\
    &= \int_{S} \int_{[0,1]} (w(u,v)-  w'(u,v))dvdu -\int_{[0,1]\setminus S}  \int_{[0,1]} (w(u,v)-  w'(u,v))dvdu
  \end{align*}
  where $S = \left\{u\in[0,1] : \int_{[0,1]} (w(u,v)-  w'(u,v))dv \geq 0 \right\}$.
  But, from the definition of $\lVert \cdot \rVert_\Box$,
  each of the two summands has modulus bounded by $\|w-w'\|_\Box$.
  We finally get
  \[d_{\mathrm{Wass}}(D_W, D_{W'}) \le 2\|w-w'\|_\Box.
  \]
  which ends the proof by definition of $\dbox$ since the choice of representatives $w,w'$ was arbitrary.
\end{proof}

Remark that when $\bm W$ is a random graphon, $D_{\bm W}$ is a random measure.
We recall, see {\em e.g.} \cite[Lemma 2.4]{RandomMeasures}, that given a random measure $\bm \mu$ on some space $B$,
its {\em intensity measure} $I[\bm \mu]$ is the deterministic measure on $B$ defined by
$I[\bm \mu](A) = \esper[\bm \mu(A)]$ for any measurable subset $A$ of $B$.

To get an intuition of what $I[D_{\bm W}]$ is for a random graphon $\bm W$,
it is useful to consider the case where $\bm W=W_{\bm G}$ is the graphon associated with a random graph $\bm G$ of size $n$.
In this case, for any measurable subset $A$ of $[0,1]$,
\[ D_{W_{\bm G}}(A)= \proba(\, \tfrac1n \deg_{\bm G}(\bm v) \in A\ |\ \bm G),\]
where $\bm v$ is a uniform random vertex in $\bm G$. Therefore
\[I[D_{W_{\bm G}}](A)=\esper\big[D_{W_{\bm G}}(A) \big] = \proba(\,  \tfrac1n \deg_{\bm G}(\bm v) \in A),\]
so that $I[D_{W_{\bm G}}]$ is the law of the normalized degree 
of a uniform random vertex $\bm v$ in the random graph $\bm G$.

We sum up the results of this section into the following proposition.
\begin{proposition}\label{prop:degree}
	Let $\bm H_n$ be a random graph of size $n$ for every $n$, and $\bm W$ be a random graphon, such that $W_{\bm H_n} \xrightarrow[n\to\infty]{d} \bm W$. Then we have the following convergence in distribution of random measures: 
	\begin{equation*}
	\frac 1 n \sum_{v\in \bm H_n} \delta_{\deg_{\bm H_n} (v)/n} \xrightarrow[n\to\infty]{d} D_{\bm W}.
	\end{equation*}
	Furthermore, denoting $\bm v_n$ a uniform random vertex in $\bm H_n$ and $\bm Z$ a random variable with law $I[D_{\bm W}]$, 
	\begin{equation*}
	\tfrac 1 n \deg_{\bm H_n} (\bm v_n) \xrightarrow[n\to\infty]{d} \bm Z .
	\end{equation*}	
\end{proposition}
\begin{proof}
From \cref{lem:cv-degree-distribution}, we immediately obtain $D_{W_{\bm H_n}} \stackrel d\to D_{\bm W}$, 
which is by definition of $D_{W_G}$ exactly the first of the stated convergences. 
	The second one follows from the first, combining Lemma 4.8 and Theorem 4.11 of \cite{RandomMeasures}\footnote{Theorem 4.11 tells us that if random measures $(\bm \xi_n)$
		converge in distribution to $\bm \xi$ then, for any compactly supported continuous function $f$,
		we have $\bm \xi_n f \xrightarrow[n\to\infty]{d}\bm \xi f$.
		But since those variables are bounded (by $\|f\|_\infty$), this convergence also holds in $L^1$,
		\emph{i.e.} $\bm \xi_n \stackrel{L^1}{\to} \bm \xi$ in the notation of \cite{RandomMeasures}.
		By Lemma 4.8, this implies the convergence of the corresponding intensity measures.}.
\end{proof}

\section{The Brownian cographon}
\label{sec:BrownianCographon}
\subsection{Construction}
Let $\Exc$ denote a Brownian excursion of length one.
We start by recalling a technical result 
on the local minima of $\Exc$: the first two assertions below are well-known,
we refer to \cite[Lemma 2.3]{MickaelConstruction} for the last one.
\begin{lemma}
	With probability one, the following assertions hold.
    First, all local minima of $\Exc$ are strict, and hence form a countable set.
    Moreover, the values of $\Exc$ at two distinct local minima are different. 
    Finally, there exists an enumeration $(b_i)_i$ of the local minima of $\Exc$,
    such that for every $i\in \mathbb N, x,y\in[0,1]$, the event $\{b_i \in (x,y), \Exc(b_i) = \min_{[x,y]}\Exc\}$ is measurable.
\end{lemma}
Let $\bm S^p = (\bm s_1,\ldots)$ be a sequence of i.i.d. random variables in $\{\Zero,\One\}$, independent of $\Exc$,
with $\proba(\bm s_1 = \Zero) = p$ (in the sequel, we simply speak of i.i.d. decorations of bias $p$).
We call $(\Exc,\bm S^p)$ a \textit{decorated Brownian excursion}, thinking of the decoration $\bm s_i$ as attached to the local minimum $b_i$.
For $x,y\in[0,1]$, we define $\Dec(x,y;\Exc,\bm S^p)$ to be
the decoration of the minimum of $\Exc$ on the interval $[x,y]$ (or $[y,x]$ if $y\le x$; we shall not repeat this precision below).
If this minimum is not unique or attained in $x$ or $y$ and therefore not a local minimum,
$\Dec(x,y;\Exc,\bm S^p)$ is ill-defined and we take the convention $\Dec(x,y;\Exc,\bm S^p)=\Zero$.
Note however that, for uniform random $x$ and $y$, this happens with probability $0$,
so that the object constructed in \cref{def:BrownianCographon} below is independent from this convention.
\begin{definition}\label{def:BrownianCographon}
	The Brownian cographon $\bm W^p$ of parameter $p$ is the equivalence class of the random function\footnote{Of course, in the image set of $\bm w^p$, 
the \emph{real values} $0$ and $1$ correspond to the \emph{decorations} $\Zero$ and $\One$ respectively.}
$$
\begin{array}{ r c c c}
\bm w^p: & [0,1]^2 &\to& \{0,1\};\\
 & (x,y) & \mapsto & \Dec(x,y;\Exc,\bm S^p).
\end{array}
$$
\end{definition}

In most of this article, we are interested in the case $p=1/2$;
in particular, as claimed in \cref{th:MainTheorem}, $W^{1/2}$ is the limit
of uniform random (labeled or unlabeled) cographs, justifying its name.

\subsection{Sampling from the Brownian cographon}

We now compute the distribution of the random graph $\Sample_k(\bm W^p)$.
\begin{proposition}
\label{prop:CaracterisationBrownianCographon}
	If $\bm W^p$ is the Brownian cographon of parameter $p$, then for every $k\geq 2$,  $\Sample_k(\bm W^p)$ is distributed 
    like the unlabeled version of $\Cograph(\bm b^p_k)$,
    where the cotree $\bm b^p_k$ is a uniform labeled binary tree with $k$ leaves equipped with i.i.d. decorations of bias $p$.
\end{proposition}
Let us note that $\bm b^p_k$ is not necessarily a canonical cotree.
\begin{proof}
We use a classical construction (see \cite[Section 2.5]{LeGall}) 
which associates to an excursion $e$ and real numbers $x_1,\cdots,x_k$ a plane tree,
denoted $\mathrm{Tree}(e;x_1,\ldots,x_k)$, which has the following properties:
\begin{itemize}
  \item its leaves are labeled with $1,\cdots,k$ and correspond to $x_1, \ldots, x_k$ respectively;
  \item its internal nodes correspond to the local minima
    of $e$ on intervals $[x_i,x_j]$;
  \item the first common ancestor of the leaves $i$ and $j$ corresponds to the local minimum of $e$ on $[x_i,x_j]$.
\end{itemize}
The tree $\mathrm{Tree}(e;x_1,\ldots,x_k)$ is well-defined with probability $1$
when $e$ is a Brownian excursion and $x_1,\cdots,x_k$ i.i.d. uniform random variables in $[0,1]$.
Moreover, in this setting, it has the distribution of a uniform random plane and labeled binary tree with $k$ leaves \cite[Theorem 2.11]{LeGall}. 
Forgetting the plane structure, it is still uniform among binary trees with $k$ labeled leaves, because the number of plane embeddings of a \emph{labeled binary tree} depends only on its size.

We now let $(\Exc,\bm S)$ be a decorated Brownian excursion, and 
$X_1, \ldots, X_k$ denote a sequence of i.i.d. uniform random variables in $[0,1]$, independent from $(\Exc,\bm S)$. 
We make use of the decorations $\bm S$ of the local minima of $\Exc$ to turn $\mathrm{Tree}(\Exc;X_1,\ldots,X_k)$ into a cotree. 
Namely, since its internal nodes correspond to local minima of $\Exc$,
we can simply report these decorations on the tree,
and we get a decorated tree $\mathrm{Tree}_{\Zero/\One}(\Exc,\bm S;X_1,\ldots,X_k)$.
When the decorations in $\bm S$ are i.i.d. of bias $p$,
then $\mathrm{Tree}_{\Zero/\One}(\Exc,\bm S,X_1,\ldots,X_k))$
is a uniform labeled binary tree with $k$ leaves, equipped with i.i.d. decorations of bias $p$.

Finally, recall that $\Sample_k(\bm W^p)$ is built by considering $X_1,\ldots,X_k$ i.i.d. uniform in $[0,1]$ and connecting vertices $v_i$ and $v_j$
if and only if $\bm w^p(X_i,X_j)=1$ (since a representative $\bm w^p$ of $\bm W^p$ takes value in $\{0,1\}$,
there is no extra randomness here). By definition of $\bm w^p$, $\bm w^p(X_i,X_j)=1$ means that the decoration of the minimum of $\Exc$
on $[X_i, X_j]$ is $\One$.
But, by construction of $\mathrm{Tree}_{\Zero/\One}(\Exc,\bm S;X_1,\ldots,X_k)$, this decoration is 
that of the first common ancestor of the leaves $i$ and $j$ in $\mathrm{Tree}_{\Zero/\One}(\Exc,\bm S;X_1,\ldots,X_k)$.
So it is equal to $\One$ if and only if $i$ and $j$ are connected in the associated cograph
(see \cref{obs:caract_edge-label1}). Summing up, we get the equality
of unlabeled random graphs
\begin{equation*}
  \Sample_k(\bm W^p) = \Cograph \big( \mathrm{Tree}_{\Zero/\One}(\Exc,\bm S,X_1,\ldots,X_k)) \big), 
  \label{eq:Sample_Wp}
\end{equation*}
ending the proof of the proposition.
\end{proof}

\subsection{Criterion of convergence to the Brownian cographon}
The results obtained so far yield 
a very simple criterion for convergence to the Brownian cographon.
For simplicity and since this is the only case we need in the present paper,
we state it only in the case $p=1/2$.
\begin{lemma}\label{lem:factorisation}
	Let $\bm t^{(n)}$ be a random cotree of size $n$ for every $n$ (which may be labeled or not).
	For $n\geq k\geq 1$, denote by $\bm t^{(n)}_k$ the subtree of $\bm t^{(n)}$
    induced by a uniform $k$-tuple of distinct leaves.
	Suppose that for every $k$ and for every labeled binary cotree $t_0$ with $k$ leaves, 
	\begin{equation}\proba(\bm t^{(n)}_k=t_0) \xrightarrow[n\to\infty]{} \frac{(k-1)!}{(2k-2)!}.\label{eq:factorisation}\end{equation}
	Then $W_{\Cograph(\bm t^{(n)})}$ converges as a graphon to $\bm W^{1/2}$.
\end{lemma}
\begin{proof} 
	We first remark that $\frac{(k-1)!}{(2k-2)!} = \frac 1 {\mathrm{card}(\mathcal C_k)}$, where $\mathcal C_k$ is the set of labeled binary cotrees with $k$ leaves. 
	Indeed the number of plane labeled binary trees with $k$ leaves is given by $k!\,\mathrm{Cat}_{k-1}$
	where $\mathrm{Cat}_{k-1}=\frac{1}{k}\binom{2k-2}{k-1}$ is the $(k-1)$-th Catalan number.
	Decorations on internal nodes induce the multiplication by a factor $2^{k-1}$
	while considering non-plane trees yields a division by the same factor in order to avoid symmetries.
	Therefore $\mathrm{card}(\mathcal{C}_k)=k!\,\mathrm{Cat}_{k-1} = \frac{(2k-2)!}{(k-1)!}$.
	
	Consequently, \cref{eq:factorisation} states that $\bm t^{(n)}_k$ converges 
    in distribution to a uniform element of $\mathcal C_k$. 
	Morever, a uniform element of $\mathcal C_k$ is distributed as $\bm b^{1/2}_k$
	where $\bm b^{1/2}_k$ is a uniform labeled binary tree with $k$ leaves equipped with i.i.d. decorations of bias ${1/2}$.
	Hence, as $n$ tends to $+\infty$,
    we have the following convergence of random labeled graphs of size $k$,
	\[\Cograph(\bm t^{(n)}_k) \stackrel d \rightarrow \Cograph(\bm b^{1/2}_k).\] 
	 Forgetting the labeling,
    the left-hand side is $\SubGraph(\vec{V'}^k,\Cograph(\bm t^{(n)}))$, 
    where $\vec{V'}^k$ is a uniform tuple of $k$ distinct vertices of $\Cograph(\bm t^{(n)})$;
    see the definition of $\bm t^{(n)}_k$ in the statement of the lemma and \cref{prop:diagramme_commutatif}. 
    Moreover, thanks to \cref{prop:CaracterisationBrownianCographon},
    forgetting again the labeling, the right-hand side has the same distribution as $\Sample_k(\bm W^{1/2})$.
    This proves the lemma, using \cref{th:Portmanteau} (namely, the implication $(d) \Rightarrow (a)$, and \cref{eq:SampleK_gK}
    together with \cref{Prop:CaracterisationLoiGraphon} to identify the limit in item (a) with $\bm W^{1/2}$). 
\end{proof}

\subsection{The degree distribution of the Brownian cographon}\label{ssec:degdistrib}

In this section we are interested in the degree distribution $D_{\bm W^p}$ of the Brownian cographon.
It turns out that, in the special case $p=1/2$,
the intensity $I[D_{\bm W^{1/2}}]$ is particularly simple.
\begin{proposition}\label{prop:IntensityW12}
$I[D_{\bm W^{1/2}}] \stackrel{d}{=} U$, where $U$ is the Lebesgue measure on $[0,1]$.
\end{proposition}
\begin{proof}
Rather than working in the continuous,
we exhibit a discrete approximation $\bm G_n^b$ of the Brownian cographon,
which has the remarkable property that the degree of a uniform random vertex $\bm v_n$
in $\bm G_n^b$ is {\em exactly} distributed as a uniform random variable in $\{0,1,\cdots,n-1\}$.

To construct $\bm G_n^b$, we
let $\bm b_n$ be a uniform $\Zero/\One$-decorated plane labeled binary tree with $n$ leaves.
Forgetting the plane structure, it is still uniform among labeled binary cotrees with $n$ leaves.
Set $\bm G_n^b=\Cograph(\bm b_n)$.
From \cref{prop:CaracterisationBrownianCographon}, $\bm G_n^b$ has the same distribution as $\Sample_n(\bm W^{1/2})$,
so that $W_{\bm G_n^b}$ converges in distribution to $\bm W^{1/2}$ (\cref{lem:sample_converges_to_W}).

Consider a uniform random vertex $\bm v_n$ in $\bm G_n^b$.
Thanks to \cref{prop:degree}, $\mathrm{Law}\big(\tfrac 1 n \deg_{\bm G_n^b}(\bm v_n)\big)$ converges to $I[D_{\bm W^{1/2}}]$.
Proving the following claim will therefore conclude the proof of the proposition.
\medskip

{\em Claim.} The law of $\deg(\bm v_n)$ in $\bm G_n^b$ is the uniform law in $\{0,1,\cdots,n-1\}$.
\smallskip

{\em Proof of the claim.}
We start by defining two operations for deterministic $\Zero/\One$-decorated plane labeled binary trees $b$.
\begin{itemize}
  \item First, we consider a (seemingly unnatural\footnote{This order is actually very natural if we interpret $b$ as the separation tree of a separable permutation (see \cite{Nous1} for the definition). 
  It is simply the value order on the elements of the permutation.}) order on the leaves of $b$.
    To compare two leaves $\ell$ and $r$, we look at their first common ancestor $u$ and
    assume w.l.o.g. that $\ell$ and $r$ are descendants of its left and right children, respectively.
    If $u$ has decoration $\Zero$, we declare $\ell$ to be smaller than $r$;
    if it has decoration $\One$, then $r$ is smaller than $\ell$. It is easy to check that this defines
    a total order on the leaves of $b$ (if we flip the left and right subtrees of internal nodes with decoration $\One$,
    this is simply the left-to-right depth-first order of the leaves).
    We write $\rank_b(\ell)$ for the rank of a leaf $\ell$ in this order.
  \item Second, we define an involution $\Phi$ on the set of $\Zero/\One$-decorated plane labeled binary trees $b$
    with a distinguished leaf $\ell$.
    We keep the undecorated structure of the tree, and simply flip the decorations of all the ancestors of $\ell$ which have $\ell$ as a descendant {\em of their right child}.
    This gives a new decorated plane labeled binary tree $b'$ and
    we set $\Phi(b,\ell)=(b',\ell)$.
\end{itemize}
Consider $b$ as above, with two distinguished leaves $\ell$ and $\tilde{\ell}$.
The corresponding vertices $v$ and $\tilde{v}$ in $G=\Cograph(b)$ are connected if and only if the first common ancestor $u$
of $\ell$ and $\tilde{\ell}$ in $b$ has decoration $\One$. Setting $\Phi(b,\ell)=(b',\ell)$,
this happens in two cases:
\begin{itemize}
  \item either $\ell$ is a descendant of the left child of $u$,
    and $u$ has decoration $\One$ in $b'$;
  \item or $\ell$ is a descendant of the right child of $u$,
    and $u$ has decoration $\Zero$ in $b'$;
\end{itemize}
This corresponds exactly to $\ell$ being bigger than $\tilde{\ell}$ in the order associated to $b'$.
Consequently, $\deg_G(v)$ is the number of leaves smaller than $\ell$ in that order, \emph{i.e.}
\begin{equation}
  \deg_G(v) =\rank_{b'}(\ell)-1. 
  \label{eq:deg_rank}
\end{equation}

Recall now that $\bm G_n^b=\Cograph(\bm b_n)$, where $\bm b_n$ is a uniform $\Zero/\One$-decorated plane labeled binary tree with $n$ leaves.
The uniform random vertex $\bm v_n$ in $\bm G_n^b$ corresponds to a uniform random leaf $\bm \ell_n$ in $\bm b_n$.
Set $(\bm b'_n,\bm \ell_n)=\Phi(\bm b_n,\bm \ell_n)$.
Since $\Phi$ is an involution, $(\bm b'_n,\bm \ell_n)$ is a uniform $\Zero/\One$-decorated plane labeled binary tree 
of size $n$ with a uniform random leaf $\bm \ell_n$.
Conditioning on $\bm b'_n$, the rank $\rank_{\bm b'_n}(\bm \ell_n)$ is a uniform random variable in $\{1,\cdots,n\}$.
The same holds taking $\bm b'_n$ at random, and,
using \cref{eq:deg_rank}, we conclude that $\deg_{\bm G_n^b}(\bm v_n)$ is
a uniform random variable in $\{0,\cdots,n-1\}$.

This proves the claim, and hence the proposition.
\end{proof}

\begin{remark}
  It seems likely that this result can also be proved by working only
  in the continuous. In particular, using a result of Bertoin and Pitman 
  \cite[Theorem 3.2]{BertoinPitman},
  the degree $D(x)=\int_y \bm W^{1/2}(x,y) dy$ of a uniform random $x$ in $[0,1]$
  in the Brownian cographon corresponds to the cumulated length of
  a half of the excursions in a Brownian bridge.
\end{remark}

%
%
%

\section{Convergence of labeled cographs to the Brownian cographon}\label{sec:proofLabeled}

In this section, we are interested in labeled cographs with $n$ vertices,
which are in one-to-one correspondence with labeled canonical cotrees with $n$ leaves (\cref{prop:canonical_cotree}).

  To study these objects, we use the framework of labeled combinatorial classes,
as presented in the seminal book of Flajolet and Sedgewick \cite[Chapter II]{Violet}.
In this framework, an object of size $n$ has $n$ atoms, which are labeled bijectively
with integers from $1$ to $n$. For us, the atoms are simply the leaves of the trees,
which is consistent with \cref{def:cotree}.
We will also consider (co)trees with marked leaves and, here,
more care is needed.
Indeed, in some instances, those marked leaves have a label
(and thus should be seen as atoms and counted in the size of the objects),
while, in other instances, they do not have a label 
(and are therefore not counted in the size of the object).
To make the distinction, we will refer to marked leaves of the latter type (\emph{i.e.} without labels)
as {\em blossoms} and reserve {\em marked leaves} for those carrying labels.

\subsection{Enumeration of labeled canonical cotrees}

Let $\mathcal{L}$ be the family of non-plane labeled rooted trees in which internal nodes have at least two  children. %
For $n\geq 1$, let $\ell_n$ be the number of trees with $n$ leaves in $\mathcal{L}$.
Let $L(z)$ denote the corresponding exponential generating function:
$$
L(z)=\sum_{n\geq 1} \frac{\ell_n}{n!}z^n.
$$

\begin{proposition}\label{prop:SeriesS}
  The series $L(z)$ is the unique formal power series without constant term solution of
\begin{equation}\label{eq:SerieS}
L(z)=z+\exp(L(z))-1-L(z).
\end{equation}
\end{proposition}

\begin{proof}%
(This series is treated in \cite[Example VII.12 p.472]{Violet}.) 
A tree in $\mathcal{L}$ consists of
\begin{itemize}
\item either a single leaf (counted by $z$) ;
\item or a root to which is attached an unordered sequence of at least two  trees of $\mathcal{L}$ (counted by $\sum_{k\geq 2}L^k/k!= e^L-1-L$).
\end{itemize}
This justifies that $L(z)$ satisfies \cref{eq:SerieS}.
The uniqueness is straightforward,
since \cref{eq:SerieS} determines for every $n$ the coefficient of $z^n$ in $L(z)$ from those of $z^k$ for $k<n$. 
\end{proof}

Computing the first coefficients, we find
$$
L(z)=z+\frac{z^2}{2!}+4\frac{z^3}{3!}+26\frac{z^4}{4!}+236\frac{z^5}{5!}+2752\frac{z^6}{6!} \mathcal{O}(z^7).
$$
These numbers correspond to the fourth Schr\"oder's problem
(see \href{http://oeis.org/A000311}{Sequence A000311} in \cite{SloaneSchroder}).\medskip

Let $m_n$ be the number of labeled canonical cotrees with $n$ leaves. We have $m_1=1$ and $m_n=2\,\ell_n$ for $n\geq 2$.
Indeed to each tree of $\mathcal{L}$ containing internal nodes (\emph{i.e.}, with at least two leaves) there correspond two canonical cotrees:
one with the root decorated by $\Zero$ and one with the root decorated by $\One$ (the other decorations are then determined by the alternation condition).
The exponential generating series $M(z)=\sum_{n\geq 1} \frac{m_n}{n!}z^n$ of labeled canonical cotrees (or equivalently of labeled cographs) 
thus satisfies $M(z)=2L(z)-z$. Combining this with \cref{prop:SeriesS}, we have 
\begin{equation}\label{eq:Lien_T_expS}
M(z)=\exp(L(z))-1.
\end{equation}

It is standard (and easy to see) that the series 
$$
L'(z)=\sum_{n\geq 1} \frac{\ell_n}{(n-1)!}z^{n-1} \quad \text{ and } \quad L^\bullet(z)=zL'(z)=\sum_{n\geq 1} \frac{\ell_n}{(n-1)!}z^{n}
$$
counts trees of $\mathcal{L}$ with a blossom or a marked leaf, repectively. 
In the subsequent analysis we need to consider the generating function
$L^\text{even}$ (resp. $L^\text{odd}$) counting
trees of $\mathcal{L}$ having a blossom at even (resp. odd) distance from the root.
Obviously, $L^\text{even}+L^\text{odd} =L'$.
\begin{proposition}\label{prop:SeriesSeven}
We have the following identities
\begin{align}
L^\text{even}&=\frac{1}{e^L(2-e^L)}, \label{eq:SevenExplicite}\\
L^\text{odd}&=\frac{e^L-1}{e^L(2-e^L)}. \label{eq:SoddExplicite}
\end{align}
\end{proposition}
\begin{proof}%
We first claim that
\begin{equation}\label{eq:SevenSoddImplicite}
\begin{cases}
L^\text{even}&=1+ (e^L-1)L^\text{odd},\\
L^\text{odd}&= (e^L-1)L^\text{even}.
\end{cases}
\end{equation}
We prove the first identity, the second one is proved similarly. 
A tree counted by $L^\text{even}$ is
\begin{itemize}
\item either reduced to a blossom (therefore the tree has size $0$, \emph{i.e.} is counted by $1$);
\item or has a root to which are attached 
\begin{itemize}
\item a tree with a blossom at odd height (counted by $L^\text{odd}$), and
\item an unordered sequence of at least one unmarked trees (counted by $\sum_{k\geq 1}L^k/k!= e^L-1$).
\end{itemize}
\end{itemize}
We obtain the proposition by solving \cref{eq:SevenSoddImplicite}.
\end{proof}

\smallskip

\subsection{Enumeration of canonical cotrees with marked leaves inducing a given cotree} 
For a labeled (not necessarily canonical) cotree $t_0$ of size $k$, we consider the family $\mathcal{M}_{t_0}$ of tuples $\left(t;\ell_1,\dots,\ell_k\right)$ where 
\begin{itemize}
\item $t$ is a labeled canonical cotree;
\item $(\ell_1,\dots,\ell_k)$ is a $k$-tuple of distinct leaves of $t$;
\item the subtree of $t$ induced by $(\ell_1,\dots,\ell_k)$ is $t_0$.
\end{itemize}
We denote by $M_{t_0}$ the associated exponential generating function.
\begin{theorem}\label{th:SerieT_t0}
Let $t_0$ be a labeled cotree with $k$ leaves. 
Denote by $n_v$ its number of internal nodes, by $n_=$ its number of edges of the form  $\Zero-\Zero$ or $\One-\One$, and by $n_{\neq}$ its number of  edges of the form  $\Zero-\One$ or $\One-\Zero$.
We have the identity
\begin{equation}\label{Eq:ComptageLabeled}
M_{t_0} = (L') (\exp(L))^{n_v} (L^\bullet)^k (L^\text{odd})^{n_=} (L^\text{even})^{n_{\neq}}.
\end{equation}
\end{theorem}

\begin{figure}[htbp]
\begin{center}
\includegraphics[width=10cm]{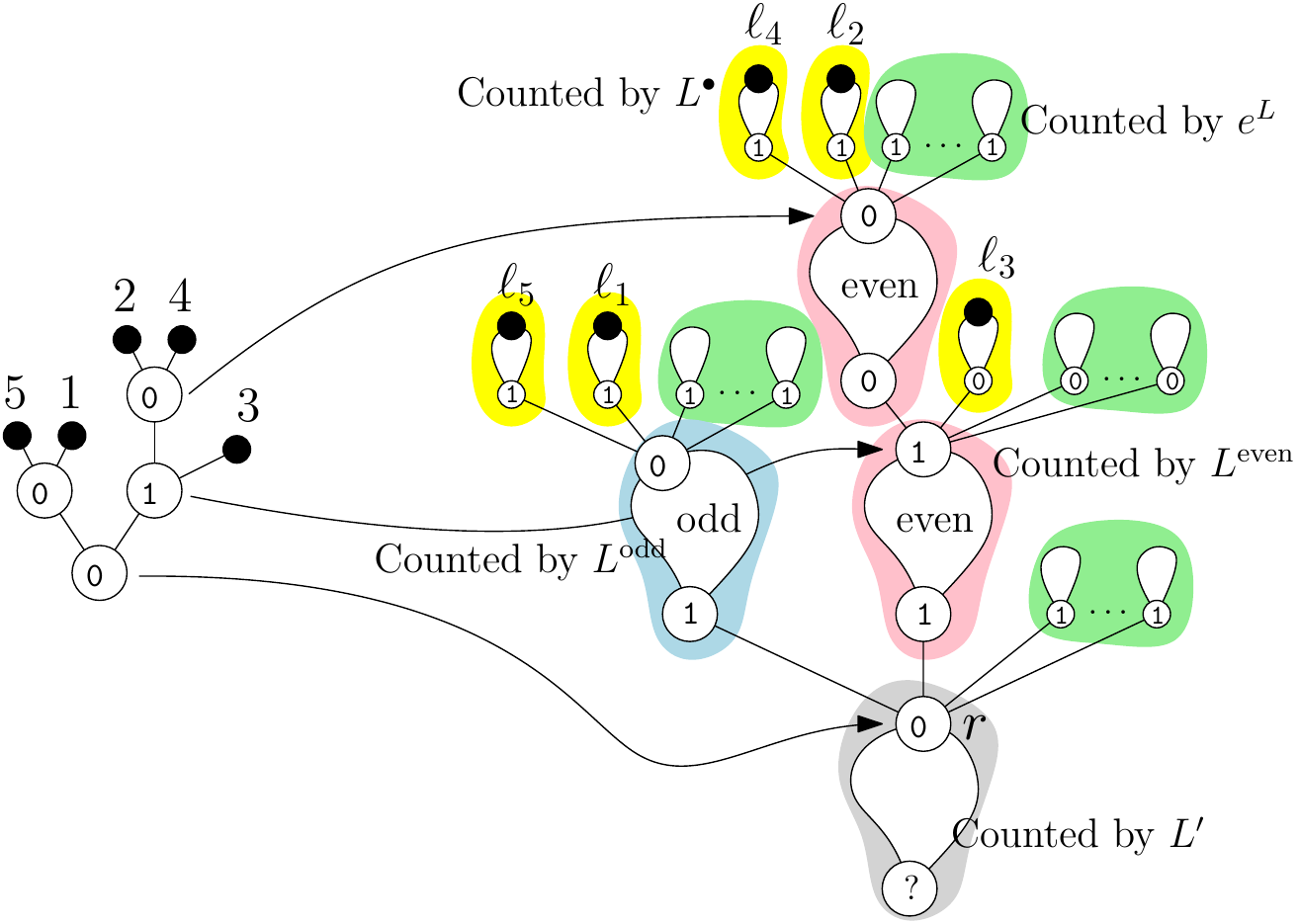}
\caption{On the left: a (non-canonical) labeled cotree $t_0$ of size $5$. On the right: a schematic view of a canonical cotree in $\mathcal{M}_{t_0}$.}
\label{fig:SqueletteStyliseCographe}
\end{center}
\end{figure}

\begin{proof}%
\emph{(Main notations of the proof are summarized in \cref{fig:SqueletteStyliseCographe}.)}

Let $\left(t;\ell_1,\dots,\ell_k\right) \in \mathcal{M}_{t_0}$.
There is a correspondence between the nodes of $t_0$ and some nodes of $t$,
mapping leaves to marked leaves and internal nodes to first common ancestors of marked leaves.
These first common ancestors of marked leaves in $t$ will be refered to as {\em branching nodes} below.
In order to prove \cref{Eq:ComptageLabeled} we will decompose each such $t$ into subtrees,
called {\em pieces}, of five different types: \emph{pink, blue, yellow, green} and \emph{gray} 
(see the color coding\footnote{We apologize to the readers to whom only a black-and-white version were available.} in \cref{fig:SqueletteStyliseCographe}). 
Our decomposition has the following property:
to reconstruct an element of $\mathcal{M}_{t_0}$,
we can choose each piece independently in a set depending on its color only,
so that the generating series of $\mathcal{M}_{t_0}$ writes as a product
of the generating series of the pieces.

In this decomposition,
there is exactly one gray piece 
obtained by pruning $t$ at the node $r$ of $t$ corresponding to the root of $t_0$. 
In this piece, $r$ is replaced by a blossom.
We note that, by definition of induced cotree, the decoration of $r$ has to match that of the root of $t_0$.
Since decorations in canonical cotrees must alternate,
this determines all decorations in the gray piece.
Possible choices for the gray piece are therefore counted
by the same series as undecorated trees with a blossom, 
that is $L'$.

For the rest of the decomposition, we 
consider branching nodes $v$ of $t$ (including $r$), and look at all children $w$ of such nodes $v$.
\begin{itemize}
 \item If such a node $w$ has exactly one descendant (possibly, $w$ itself) which is a marked leaf, 
 we build a piece, colored yellow, by taking the fringe subtree rooted at $w$. 
   Yellow pieces are labeled canonical cotrees with one marked leaf. 
   However, the decoration within the yellow piece is forced by the alternation of decorations
    in $t$ and by the decoration of the parent $v$ of $w$, which has to match the decoration
    of the corresponding node in $t_0$ (see \cref{fig:SqueletteStyliseCographe}). 
   So the generating function for yellow pieces is $L^\bullet$.

   Of course, we have a yellow piece for each marked leaf of $t$, \emph{i.e.} for each leaf of $t_0$. 

 \item If a node $w$ child of a branching node in $t$ has at least two marked leaves among its descendants, 
 it must also have a descendant (possibly equal to $w$) that is a branching node.
 We define $v'$ as the branching node descending from $w$ (possibly equal to it) which is the closest to $w$. 
 This implies that the node of $t_0$ corresponding to $v'$ (denoted $v'_0$) is a child of the one corresponding to $v$ (denoted $v_0$). 
 We build a piece rooted at $w$, which corresponds to the edge $(v_0,v'_0)$ of $t_0$. 
 This piece is the fringe subtree rooted at $w$ pruned at $v'$,
    {\em i.e.} where $v'$ is replaced by a blossom.
    We color it blue if the blossom is at odd distance from $w$, pink otherwise.
    The generating functions for blue and pink pieces are therefore $L^{\text{odd}}$ and $L^{\text{even}}$,
    respectively (since again all decorations in the piece are dertermined by the one of $v_0$).
    
    Because of the alternation of decoration, the piece is blue if and only if
    $w$ and $v'$ have different decorations in $t$,
    which happens if and only if $v$ and $v'$ (or equivalently, $v_0$ and $v'_0$) have the same decoration.
    We therefore have a blue piece for each internal edge of $t_0$ with extremities 
    with the same decoration, and a pink piece for each internal edge of $t_0$
    with extremities with different decorations.

\item All other nodes $w$ have no marked leaf among their descendants. 
We group all such nodes $w$ that are siblings to build a single green piece, 
attached to their common parent $v$. 
Namely, for each branching node $v$, we consider all its children $w$ having no marked leaf as a descendant (possibly, there are none), 
and we define the green piece attached to $v$ as the (possibly empty) forest of fringe subtrees of $t$ rooted at these nodes $w$.
  Green pieces are forest, \emph{i.e.} unordered set of labeled canonical cotrees.
    The decoration within the green piece is forced by the alternation of decorations
    in $t$ and by the decoration of $v$, which as before has to match the decoration
    of the corresponding node in $t_0$.
    Therefore, choosing a green piece amounts to choosing an unordered set of undecorated trees in $\mathcal L$.
    We conclude that possible choices for each green piece are counted by $e^L$. 
    
    Finally, we recall that there is one (possibly empty) green piece for each branching node of $t$, \emph{i.e.}
    for each internal node of $t_0$.    
\end{itemize}

Since $t_0$ is a labeled cotree, leaves / internal nodes / edges of $t_0$ can be ordered in a canonical way.
Since yellow (resp. green, resp. blue and pink) pieces in the above decomposition are indexed by leaves (resp. internal nodes, resp. edges) of $t_0$,
they can be ordered in a canonical way as well.
Moreover, the correspondence between marked trees $(t;\ell_1,\cdots,\ell_k)$
in $\mathcal{M}_{t_0}$ and tuples of colored pieces is one-to-one.
This completes the proof of \cref{Eq:ComptageLabeled}.
\end{proof}
\subsection{Asymptotic analysis}
  Following Flajolet and Sedgewick \cite[p. 389]{Violet}, we say that a power series is $\Delta$-analytic
if it is analytic in some $\Delta$-domain $\Delta(\phi,\rho)$, where $\rho$ is its radius of convergence.
This is a technical hypothesis, which enables to apply the transfer theorem \cite[Corollary VI.1 p. 392]{Violet};
all series in this paper are $\Delta$-analytic.%
\begin{proposition}\label{prop:Asympt_S_Seven}
  The series $L(z)$ has radius of convergence $\rho=2\log(2)-1$ and is $\Delta$-analytic.
Moreover, the series $L$ is convergent at $z=\rho$ and we have
\begin{equation}\label{eq:Asympt_S}
L(z)\underset{z\to \rho}{=}\log(2)-\sqrt{\rho}\sqrt{1-\tfrac{z}{\rho}} +\mathcal{O}(1-\tfrac{z}{\rho}).
\end{equation}
\end{proposition}

\begin{proof}%
Using \cref{prop:SeriesS}, \cref{prop:Asympt_S_Seven} is a direct application of \cite[Theorem 1]{MathildeMarniCyril}. 
\end{proof}

It follows from \cref{prop:Asympt_S_Seven} 
that $L'$, $\exp(L)$, $L^{\text{even}}$ and $L^{\text{odd}}$ also have radius of convergence $\rho=2\log(2)-1$,
are all $\Delta$-analytic
and that their behaviors near $\rho$ are 
\begin{equation}
  L'(z) \underset{z\to \rho}{\sim} \frac{1}{2\sqrt{\rho}}\left(1-\tfrac{z}{\rho}\right)^{-1/2} \label{eq:asymLprime};
  \qquad   \exp(L(z))  \underset{z\to \rho}{\sim} 2; %
  \end{equation}
    \begin{equation}
  L^{\text{even}}(z)  \underset{z\to \rho}{\sim} \frac{1}{4\sqrt{\rho}}\left(1-\tfrac{z}{\rho}\right)^{-1/2};  %
  \qquad
L^{\text{odd}}(z)  \underset{z\to \rho}{\sim} \frac{1}{4\sqrt{\rho}}\left(1-\tfrac{z}{\rho}\right)^{-1/2}.  \label{eq:asymLodd}
\end{equation}
Indeed, the first estimate follows from \cref{eq:Asympt_S} by singular differentiation \cite[Thm. VI.8 p. 419]{Violet}, 
while the third and fourth ones are simple computations using \cref{eq:SevenExplicite} and \cref{eq:SoddExplicite}.

\medskip

\subsection{Distribution of induced subtrees of uniform cotrees}
We take a uniform labeled canonical cotree $\bm t^{(n)}$ with $n$ leaves. 
We also choose uniformly at random a $k$-tuple $(\bm\ell_1,\cdots,\bm\ell_k)$ of distinct leaves of $\bm t^{(n)}$.
Equivalently, $(\bm t^{(n)};\bm\ell_1,\cdots,\bm\ell_k)$ is chosen uniformly at random 
among labeled canonical cotrees of size $n$ with $k$ marked leaves.
We denote by $\mathbf{t}^{(n)}_k$ the labeled cotree induced by the $k$ marked leaves.

\begin{proposition}\label{prop:proba_arbre}
Let $k\geq 2$, and let $t_0$ be a labeled binary cotree with $k$ leaves. Then
\begin{equation}\label{eq:proba_asymptotique_t0}
\proba (\mathbf{t}^{(n)}_k = t_0) \underset{n\to+\infty}{\to} 
\displaystyle{\frac{(k-1)!}{(2k-2)!}}. %
\end{equation}
\end{proposition}

\begin{proof}
We fix a labeled binary cotree  $t_0$ with $k$ leaves.
From the definitions of $\mathbf{t}^{(n)}_k$, $M$ and $M_{t_0}$ we have for $n \geq k$ (we use the standard notation $[z^n]A(z)$ for the $n$-th coefficient of a power series $A$)
\begin{equation}
  \proba(\mathbf{t}^{(n)}_k=t_0) = \frac{n! [z^{n}] M_{t_0}(z)}{n \cdots (n-k+1)\, n! [z^n] M(z)}.
  \label{eq:probat0_quotient}
\end{equation}
Indeed, the denominator counts the total number of labeled canonical cotrees $(t;\ell_1,\cdots,\ell_k)$
of size $n$ with $k$ marked leaves.
The numerator counts those tuples, for which $(\ell_1,\cdots,\ell_k)$ induce
the subtree $t_0$. The quotient is therefore the desired probability.

By \cref{th:SerieT_t0}, and using the notation introduced therein, we have
$$
M_{t_0}=  (L') (\exp(L))^{n_v} (L^\bullet)^k (L^\text{odd})^{n_=} (L^\text{even})^{n_{\neq}}.
$$
Since $t_0$ is binary, we have $n_v=k-1$ and $n_= +n_{\neq}=k-2$.
We now consider the asymptotics around $z=\rho$.
Using \cref{eq:asymLprime} and \eqref{eq:asymLodd} and recalling that $L^\bullet(z)=zL'(z)$, we get

\begin{align*}
M_{t_0}(z)
&\underset{z\to \rho}{\sim} \rho^k \left(\frac{1}{2\sqrt{\rho}}\left(1-\tfrac{z}{\rho}\right)^{-1/2}\right)^{k+1} 2^{k-1} \left(\frac{1}{4\sqrt{\rho}}\left(1-\tfrac{z}{\rho}\right)^{-1/2}\right)^{k-2}\\
&\underset{z\to \rho}{\sim} \frac{{\rho}^{1/2}}{2^{2k-2}} \left(1-\tfrac{z}{\rho}\right)^{-(k-1/2)}.
\end{align*}
By the transfer theorem (\cite[Corollary VI.1 p.392]{Violet}) we obtain 
$$
[z^{n}] M_{t_0}(z) \underset{n\to +\infty}{\sim} \frac{{\rho}^{1/2}}{2^{2k-2}\rho^{n}} \frac{n^{k-3/2}}{\Gamma(k-1/2)}
= \frac{(k-1)!}{\sqrt{\pi}(2k-2)!} \frac{n^{k-3/2}}{\rho^{n-1/2}}.
$$
Applying again the transfer theorem to $M(z)=2L(z)-z$ whose asymptotics is given in \cref{eq:Asympt_S}, we also have
$$
n(n-1)\dots (n-k+1) [z^n] M(z) \underset{n\to +\infty}{\sim}
n^k (-2\sqrt{\rho})\frac{n^{-3/2}}{\rho^n \Gamma(-1/2)}
\sim  \frac{n^{k-3/2}}{\rho^{n-1/2} \sqrt{\pi}}.
$$
Finally,
$
\proba(\mathbf{t}^{(n)}_k=t_0) \to   \frac{(k-1)!}{(2k-2)!}.
$
\end{proof}

\subsection{Proof of \cref{th:MainTheorem,th:DegreeRandomVertex} in the labeled case}\label{SousSec:PreuveEtiquete}
\ 
Since labeled canonical cotrees and labeled cographs are in bijection, $\Cograph(\bm t^{(n)})$ is a uniform labeled cograph of size $n$, \emph{i.e.} is equal to $\bm G_n$ in distribution.
 Thus \cref{th:MainTheorem} follows from \cref{lem:factorisation} and \cref{prop:proba_arbre}.
 \cref{th:DegreeRandomVertex} is now a consequence of  \cref{th:MainTheorem},
 combined with \cref{prop:degree,prop:IntensityW12}.

\section{Convergence of unlabeled cographs to the Brownian cographon}
\label{sec:unlabeled}

\subsection{Reducing unlabeled canonical cotrees to labeled objects}
\label{ssec:unlabeled_to_labeled}
In this section, we are interested in unlabeled cographs. They are in one-to-one correspondence with unlabeled canonical cotrees. 
We denote by $\overline{\mathcal V}$ the class of unlabeled canonical cotrees and by $\overline{\mathcal U}$ the class of rooted non-plane unlabeled trees with no unary nodes, counted by the number of leaves.
If $V$ and $U$ are their respective ordinary generating functions, then clearly, $V(z) = 2U(z)-z$.

The class $\overline{\mathcal U}$ may be counted using the multiset construction and the P\'olya exponential \cite[Thm. I.1]{Violet}:
a tree of $\overline{\mathcal U}$ is either a single leaf or a multiset of cardinality at least $2$ of trees of $\overline{\mathcal U}$,
yielding the following equation:
\begin{equation}
\label{eq:U}
U(z) = z + \exp\left( \sum_{r\geq 1} \frac 1 r\ U(z^r) \right) -1 - U(z).
\end{equation}

As in the labeled case, we want to count the number of pairs $(t,I)$ where $t$ is a cotree of $\overline{\mathcal V}$ with $n$ leaves, and $I$ is a $k$-tuple of leaves of $t$ 
(considered labeled by the order in which they appear in the tuple), 
such that the subtree induced by $I$ in $t$ is a given labeled cotree $t_0$. 
\medskip

To that end, we would need to refine \cref{eq:U} to count trees with marked leaves, inducing a given subtree,
in a similar spirit as in \cref{th:SerieT_t0}.
There is however a major difficulty here, which we now explain.
There are two ways of looking at tuples of marked leaves in unlabeled trees.
\begin{itemize}
\item We consider pairs $(t,I)$, where $t$ is a labeled tree and $I$ a $k$-tuple
of leaves of $t$. Then we look at orbits $\overline{(t,I)}$ of such pairs under the natural relabeling action.
\item Or we first consider orbits $\overline{t}$ of labeled trees under the relabeling action,
{\em i.e.} unlabeled trees.
For each such orbit we fix a representative and consider pairs $(\overline{t},I)$,
where $I$ is a $k$-tuple of leaves of the representative of $\overline{t}$.
\end{itemize}
In the second model, every unlabeled tree has exactly $\binom{n}{k}$ marked versions,
which is not the case in the first model\footnote{\emph{E.g.}, the tree with three leaves all attached to the root, two of which are marked, has only one marked version in the first model.}.
Consequently, if we take an element uniformly at random in the second model,
the underlying unlabeled tree is a uniform unlabeled tree,
while this property does not hold in the first model.

Our goal is to study uniform random unlabeled cographs of size $n$, where we next choose a uniform random $k$-tuple of leaves. 
This corresponds exactly to the second model.

The problem is that combinatorial decompositions of unlabeled combinatorial classes
are suited to study the first model (unlabeled objects are orbits of labeled objects under relabeling).
In particular, \cref{th:SerieT_t0} has an easy analogue for counting unlabeled trees with marked leaves 
inducing a given labeled cotree in the first sense, but not in the second sense.

To overcome this difficulty, we consider the following labeled combinatorial class:
\[ {\mathcal U} = \{ (t,a) : t\in \mathcal L, a \text{ a root-preserving automorphism of }t\}\]
where $\mathcal{L}$ is the family of non-plane labeled rooted trees in which internal nodes have at least two children,
studied in \cref{sec:proofLabeled}.
We define the size of an element $(t,a)$ of ${\mathcal U}$ as the number of leaves of $t$.
This set is relevant because of the following easy but key observation.
\begin{proposition}%
	Let $\Phi$ denote the operation of forgetting both the labels and the automorphism. 
	Then, $\Phi({\mathcal U})= \overline {\mathcal U}$ and
	every $t\in  \overline {\mathcal U}$ of size $n$ has exactly $n!$ preimages by $\Phi$.
	As a result, the ordinary generating series $U\!$ of $\,\overline {\mathcal U}$
	 equals the exponential generating function of ${\mathcal U}$
	and the image by $\Phi$ of a uniform random element of size $n$ in ${\mathcal U}$ is a uniform random element of size $n$ in $ \overline {\mathcal U}$.
\end{proposition}
\begin{proof}
The number of preimages of $t\in \overline {\mathcal U}$ is the number of automorphisms of $t$ times the number of distinct labelings of $t$, which equals $n!$ by the orbit-stabilizer theorem. The other claims follow immediately.
\end{proof}

Working with ${\mathcal U}$ instead of $ \overline {\mathcal U}$ solves
the issue raised above concerning marking, since we have labeled objects.
However the additional structure (the automorphism) has to be taken into account 
in combinatorial decompositions, but this turns out to be tractable (at least asymptotically).

\subsection{Combinatorial decomposition of $\mathcal{U}$}

We first describe a method for decomposing pairs $(t,a)$ in $\mathcal U$ at the root of $t$, 
which explains \emph{combinatorially}
why the exponential generating function $U$ of $\mathcal U$ satisfies \cref{eq:U}.
This combinatorial interpretation of \cref{eq:U} is necessary for the refinement
with marked leaves done in the next section.

Let $(t,a)\in {\mathcal U}$. Then $t$ is a non-plane rooted labeled tree with no unary nodes and $a$ is one of its root-preserving automorphisms. Assuming $t$ is of size at least two, we denote by $v_1,\ldots v_d$ the children of the root, and $t_1,\ldots, t_d$ the fringe subtrees rooted at these nodes, respectively.

Because $a$ is a root-preserving automorphism, it preserves the set of children of the root, hence there exists a permutation $\pi\in \Sn_d$ such that $a(v_i) = v_{\pi(i)}$ for all $1\leq i \leq d$.
Moreover, we have necessarily $a(t_i) = t_{\pi(i)}$ for all $1\leq i \leq d$.

Let $\pi = \prod_{s=1}^p c_s$ be the decomposition of $\pi$ into disjoint cycles, including cycles of length one.
Let $c_s = (i_1,\ldots,i_r)$ be one of them. 
We consider the forest $t(c_s)$ formed by the trees $t_{i_1},\ldots,t_{i_r}$. 
Then the pair $(t(c_s),a_{|t(c_s)})$ lies in the class $\mathcal C_r$
of pairs $(f,a)$, where $f$ is a forest of $r$ trees and $a$ an automorphism
of $f$ acting cyclically on the components of $f$.

The tree $t$ can be recovered by adding a root to $\biguplus_{s=1}^p t(c_s)$.
Moreover, $a$ is clearly determined by $(a_{|t(c_s)})_{1\leq s\leq p}$.
So we can recover $(t,a)$ knowing  $(t(c_s),a_{|t(c_s)})_{1\leq s\leq p}$.
Recall that the cycles $c_s$ indexing the latter vector are the cycles of the permutation $\pi$,
which has size at least $2$ (the root of $t$ has degree at least $2$).
Since permutations $\pi$ are sets of cycles,
we get the following decomposition of $\mathcal U$
(using as usual $\mathcal Z$ for the atomic class,
representing here the single tree with one leaf):
\begin{equation}
  \label{eq:Dec_U_Cr}
  \mathcal U= \mathcal Z \, \uplus \, \textsc{Set}_{\ge 1}\Big(\biguplus_{r\geq 1} \,\mathcal C_r\Big) \setminus \mathcal C_1,
\end{equation}

We then relate $\mathcal C_r$ to $\mathcal U$ to turn \cref{eq:Dec_U_Cr} into a recursive equation.
Let $(f,a)$ be an element of $\mathcal C_r$,
and $t$ be one of the component of $f$.
We write $f=\{t_1,\cdots,t_r\}$
such that $t_1=t$ and $a$ acts on these components by
$t_1 \overset a \to t_2 \overset a \to \cdots \overset a \to t_r \overset a \to t_1$
(this numbering of the components of $f$ is uniquely determined by $t$).
We then encode $(f,a)$
by a unique tree $\widehat t$ isomorphic to $t_1$, with multiple labelings, 
\emph{i.e.} each leaf $\widehat v\in \widehat t$, corresponding to $v\in t_1$, is labeled by
$(v,a(v), a^2(v),\ldots, a^{r-1}(v))$.
Finally, $a^r$ induces an automorphism of $\widehat t$.
Consequently, $(\widehat t, a^r)$ is an element of the combinatorial class $\mathcal U \circ \mathcal Z^r$, 
\emph{i.e.} an element of $\mathcal U$ %
where each atom (here, each leaf of the tree) carries a vector of $r$ labels;
the size of an element of  $\mathcal U \circ \mathcal Z^r$ is the total number of labels,
\emph{i.e.} $r$ times the number of leaves of $\widehat t$.
The forest $f$ and its marked component $t$ are trivially recovered from $(\widehat t, a^r)$. 
Since a forest automorphism is determined by its values on leaves,
 we can recover $a$ as well. 

 This construction defines a size-preserving bijection between triples $(f,a,t)$,
 where $(f,a)$ is in $\mathcal C_r$ and $t$ one of the component of $f$,
 and elements of $\mathcal U \circ \mathcal Z^r$.
 Forgetting the marked component $t$, this defines an $r$-to-$1$ size-preserving correspondence
 from $\mathcal C_r$ to $\mathcal U \circ \mathcal Z^r$.
 Together with \cref{eq:Dec_U_Cr},
this gives the desired combinatorial explanation to the fact
that the exponential generating function of $\mathcal U$ satisfies \cref{eq:U}.
\medskip

We now introduce the combinatorial class $\mathcal D$ of trees in $\mathcal U$
with size $\geq 2$ such that no child of the root is fixed by the automorphism.
This means that there is no cycle of size $1$ in the above decomposition of $\pi$ into cycles.
 Therefore, the exponential generating function of $\mathcal D$ satisfies 
\begin{equation}
\label{eq:D}
D(z) = \exp\left( \sum_{r\geq 2} \frac 1 r\ U(z^r) \right) -1 .%
\end{equation}
Note that introducing the series $D$ is classical when applying the method of singularity analysis on
unlabeled unrooted structures (aka P\'olya structures), see, {\em e.g.}, \cite[p 476]{Violet}. 
However, interpreting it combinatorially with objects of $\mathcal D$ is not standard, but necessary for our purpose. 

In the sequel, for $k\geq 0$, we write $\exp_{\geq k}(z) = \sum_{z\geq k} \frac {z^k}{k!}$.
Algebraic manipulations from \cref{eq:U} allow to rewrite the equation for $U$ as
\begin{equation}
\label{eq:U2}
U = z + \exp_{\geq 2}(U)+ D\exp(U).
\end{equation}
Moreover, \cref{eq:U2} has a combinatorial interpretation. 
Indeed, pairs $(t,a)$ in $\mathcal U$ of size at least $2$ can be split into two families as follows.
\begin{itemize}
\item The first family consists in pairs $(t,a)$,
for which all children of the root are fixed by the automorphism $a$; 
adapting the above combinatorial argument, we see that the generating series
of this family is $\exp_{\geq 2}(U)$ (recall that the root has at least 2 children).
\item The second family consists in pairs $(t,a)$,
where some children of the root are moved by the automorphism $a$.
Taking the root, its children moved by $a$ and their descendants give a tree $t_1$
such that $(t_1,a_{|t_1})$ is in $\mathcal D$.
Each child $c$ of the root fixed by $a$ with its descendants form a tree $t_c$
such that $(t_c,a_{|t_c})$ is in $\mathcal U$.
We have a (possibly empty) unordered set of such children. 
Therefore, elements in this second family are described as pairs consisting of 
an element of $\mathcal D$ and a (possibly empty) unordered set of elements of $\mathcal U$, 
so that the generating series of this second family is $D\exp(U)$. 
\end{itemize}
Bringing the two cases together, we obtain a combinatorial interpretation of \cref{eq:U2}.
Again, this combinatorial interpretation will be important later, when refining with marked leaves.
\medskip

We can now turn to defining the combinatorial classes that will appear in our decomposition. 
Similarly to the case of labeled cographs, we will need to consider objects of $\mathcal U$ 
(recall that these are \emph{labeled} objects) where some leaves are marked. 
Here again, we need to distinguish 
marked leaves carrying a label
(contributing to the size of the objects), 
and leave not carrying any label 
(not counted in the size).
We keep the terminology of our section on labeled cographs, namely 
we call \emph{blossoms} marked leaves of the latter type (\emph{i.e.} without labels)
and we reserve {\em marked leaves} for those carrying labels.

We let $\mathcal{U}^\bullet$ (resp. $\mathcal U'$) be the combinatorial class of pairs $(t,a)$ in $\mathcal U$ 
with a marked leaf (resp. blossom) in $t$.
Their exponential generating functions are respectively $zU'(z)$ and $U'(z)$ (the derivative of $U(z)$). 
We also define $\mathcal U^\star \subset \mathcal U'$ as the class of pairs $(t,a)$ in $\mathcal U$ 
with a blossom in $t$ {\em which is fixed by $a$}. 
Finally, we decompose $\mathcal U^\star$ as $\mathcal U^\star = \mathcal U^{\even} \uplus \mathcal U^{\odd}$, according to the parity of the distance from the root to the blossom. 
We denote by $U^\star$, $U^{\even}$ and $U^{\odd}$ the exponential generating functions of these classes, respectively. 
\begin{proposition}
	We have the following equations:
	\begin{align}
	& \hspace*{1cm} U^\star \, \, = \ 1 + U^\star\exp_{\geq 1}(U)+ U^\star D\exp(U), \label{eq:Ustar}\\
	& \begin{cases}
	U^{\even} &= \ 1 + U^{\odd}\exp_{\geq 1}(U)+ U^{\odd} D\exp(U),\\
	U^{\odd} &= \ U^{\even}\exp_{\geq 1}(U)+ U^{\even} D\exp(U). \label{eq:Uevenodd}
	\end{cases}
	\end{align}
\end{proposition}
\begin{proof}
Note that if a blossom is required to be fixed by the automorphism, then all of its ancestors are also fixed by the automorphism. 
Then, the equation of $U^\star$ is obtained by the same decomposition as for \cref{eq:U2}, imposing that the blossom belongs to a subtree attached to a child of the root which is fixed by the automorphism. 
The other two equations follow immediately.
\end{proof}

\subsection{Enumeration of canonical cotrees with marked leaves inducing a given cotree}
We first define ${\mathcal V}$ as the class of pairs $(t,a)$,
where $t$ is a labeled canonical cotree and $a$ a root-preserving automorphism of $t$.
As for ${\mathcal U}$ and $ \overline {\mathcal U}$,
we have a $n!$-to-1 map from ${\mathcal V}$ to $ \overline {\mathcal V}$
and $V$ can be seen either as the ordinary generating function of $ \overline {\mathcal V}$
or the exponential generating function of $\mathcal V$.%

We would like to find a combinatorial decomposition
of  pairs in $\mathcal V$ with marked leaves inducing a given cotree.
It turns out that it is simpler and sufficient for us to work with a smaller class,
which we now define. 
\begin{definition}
	Let $t_0$ be a labeled cotree of size $k$. Let $\mathcal V_{t_0}$ be the class 
	of tuples $(t,a;\ell_1,\dots,\ell_k)$, where $(t,a)$ is in $\mathcal V$
	and $\ell_1$, \dots, $\ell_k$ are distinct leaves of $t$ (referred to as {\em marked leaves})
	such that
	\begin{itemize}
		\item the marked leaves induce the subtree $t_0$;
		\item the following nodes are fixed by $a$: 
		all first common ancestors of the marked leaves, and their children leading to a marked leaf.
	\end{itemize}
	\label{Def:V_t_0}
\end{definition}
We note that, because of the second item in the above definition,
not all tuples $(t,a;\ell_1,\dots,\ell_k)$ (where $(t,a)$ is in $\mathcal V$
and $\ell_1$, \dots, $\ell_k$ are leaves of $t$) belong to some $\mathcal V_{t_0}$.
However, we will see below (as a consequence of \cref{prop:proba_arbre_unlabeled})
that asymptotically almost all tuples $(t,a;\ell_1,\dots,\ell_k)$ do belong to some $\mathcal V_{t_0}$
(even if we restrict to binary cotrees $t_0$, which is similar to the previous section).

Let $V_{t_0}$ be the exponential generating series of $\mathcal V_{t_0}$;
it is given by the following result.
\begin{theorem}\label{th:V_t_0}
	Let $t_0$ be a labeled cotree with $k$ leaves, $n_v$ internal nodes, $n_=$ edges of the form  $\Zero-\Zero$ or $\One-\One$, $n_{\neq}$  edges of the form  $\Zero-\One$ or $\One-\Zero$.
	We have the identity
	$$
	V_{t_0}=(U^\star)(2U+1-z)^{n_v} (U^\bullet)^k (U^\odd)^{n_=} (U^\even)^{n_{\neq}}.
	$$
\end{theorem}

\begin{proof}
Let $(t,a;\ell_1,...,\ell_k)$ be a tree in $\mathcal V_{t_0}$. 
The tree $t$ with its marked leaves $\ell_1,...,\ell_k$
can be decomposed in a unique way as in the proof of \cref{th:SerieT_t0}
 into pieces: pink trees, blue trees, yellow trees, gray trees and green forests.

As soon as a node of $t$ is fixed by the automorphism $a$, then the set of its descendants is stable by $a$.  
Therefore, the second item of \cref{Def:V_t_0} ensures that each colored piece
in the decomposition of $t$ is stable by $a$, 
so that $a$ can be decomposed uniquely into a collection of automorphisms, 
one for each colored piece.
Consequently, from now on, we think at pieces as trees/forests \emph{with an automorphism}.

As in \cref{th:SerieT_t0}, each piece can be chosen independently in a set
depending on its color. Moreover, since $t_0$ is labeled, the pieces can be ordered in a canonical way,
so that the generating series of $V_{t_0}$ is the product of the generating series of the pieces.
\begin{itemize}
\item The gray subtree is a tree with an automorphism and a blossom which is fixed by the automorphism
(because of the second item of \cref{Def:V_t_0}).
As in \cref{th:SerieT_t0}, the decoration is forced by the context, so
that we can consider the gray subtree as not decorated.
The possible choices for the gray subtrees are therefore counted by $U^\star$. 
\item The possible choices for each green forest (and its automorphism)
 are counted by $1+U+(U-z)$: the first term corresponds to the empty green piece, 
the second one to exactly one tree in the green forest, and the third one to a set 
of at least two green trees (which can be seen as a non-trivial tree in $\mathcal U$ by adding a root).
\item The possible choices for each yellow piece are counted by $U^\bullet$,
since these trees have a marked leaf which is not necessarily fixed by the automorphism. 
\item The possible choices for each pink piece are counted by  $U^{\text{even}}$:
the blossom must be at even distance from the root of the piece 
(for the same reason as in  \cref{th:SerieT_t0}) and must be fixed by the automorphism
(because of the second item of \cref{Def:V_t_0}).
\item Similarly, the possible choices for each blue piece are counted by  $U^{\text{odd}}$.
\end{itemize}
Bringing everything together gives the formula in the theorem.
\end{proof}

\subsection{Asymptotic analysis}
Let $\rho$ be the radius of convergence of $U$.
It is easily seen that we have $0<\rho<1$, see, {\em e.g.}, \cite{GenitriniPolya}, where
the numerical approximation $\rho \approx 0.2808$ is given.
\begin{proposition}\label{prop:Asympt_U}
	The series $U,U',U^\star,U^\even, U^\odd$ all have the same radius of convergence $\rho$,
	are $\Delta$-analytic and admit the following expansions around $\rho$:
	\begin{equation*}
	U(z)\underset{z\to \rho}{=}\frac{1+\rho}{2}-\beta\sqrt{\rho-z} +o(\sqrt{\rho - z}), \qquad 
	U'(z)\underset{z\to \rho}{\sim}\frac {\beta}{2\sqrt{\rho-z}},
        \end{equation*}
        \begin{equation*}
	2 U^\even(z) \sim 2U^\odd(z) \sim U^\star(z)\underset{z\to \rho}{\sim}\frac {1}{2\beta\sqrt{\rho-z}}, 
	\end{equation*}
	for some constant $\beta>0$.
\end{proposition}
To prove the proposition, we need the following lemma,
which is standard in the analysis of P\'olya structures.
\begin{lemma}\label{lem:Rayon_U} 
The radius of convergence of $D$ is $\sqrt{\rho}>\rho$.
\end{lemma}

\begin{proof}
Since $U$ has no constant term, for every $x \geq 1$ and $0<z<1$ we have $U(z^x)\leq U(z)z^{x-1}$. 
Hence for $0<t<\rho$,
\[D(\sqrt{t}) = \exp_{\geq 1}\left( \sum_{r\geq 2} \frac 1 r\ U(t^{r/2}) \right) \leq \exp\left( \sum_{r\geq 2}  U(t)t^{r/2-1} \right)\leq \exp\left(U(t) \frac 1{1-\sqrt{t}}\right) <\infty.\]
This implies that the radius of convergence of $D$ is at least $\sqrt{\rho}$.
 Looking at \cref{eq:D}, we see that $D$ termwise dominates
 $\tfrac 12 U(z^2)$, whose radius of convergence is $\sqrt{\rho}$.
 Therefore, the radius of convergence of $D$ is exactly $\sqrt{\rho}$.
\end{proof}
\begin{proof}[Proof of \cref{prop:Asympt_U}]
	Set $F(z,u) = z+ \exp_{\geq 2}(u) + D(z)\exp(u)$. Then $U$ verifies the equation $U = F(z,U)$, which is the setting of \cite[Theorem A.6]{Nous3}\footnote{We warn the reader that the function $U$ appearing in \cite[Theorem A.6]{Nous3} is unrelated to the quantity $U(z)$ in the present article (which corresponds instead to $Y(z)$ in \cite[Theorem A.6]{Nous3}).} (in the one dimensional case, which is then just a convenient rewriting of \cite[Theorem VII.3]{Violet}).
	The only non-trivial hypothesis to check is the analyticity of $F$ at $(\rho,U(\rho))$. This holds because $\exp$ has infinite radius of convergence, while $D$ has radius of convergence $\sqrt{\rho}>\rho$ from \cref{lem:Rayon_U}.
	
	From items vi) and vii) of \cite[Theorem A.6]{Nous3}, we have that $U$ and $(1-\partial_uF(z,U(z)))^{-1}$
	have radius of convergence $\rho$, are $\Delta$-analytic and that $\partial_uF(\rho,U(\rho))=1$.
	Moreover, 
	\begin{equation*}
	U(z)\underset{z\to \rho}{=}U(\rho)-\frac {\beta}{\zeta}\sqrt{\rho-z} +o(\sqrt{\rho - z}), \qquad
	U'(z)\underset{z\to \rho}{\sim}\frac {\beta}{2\zeta\sqrt{\rho-z}}
        \end{equation*}
        \begin{equation*}
	(1-\partial_uF(z,U(z)))^{-1} \underset{z\to \rho}{\sim}\frac {1}{2\beta\zeta\sqrt{\rho-z}}, 
	\end{equation*}
	where $\beta = \sqrt{\partial_zF(\rho,U(\rho))}$ and $\zeta = \sqrt{\tfrac 12 \partial^2_uF(\rho,U(\rho))}$.

	We have $\partial_uF(z,u) = \exp_{\geq 1}(u) + D(z)\exp(u) = F(z,u) + u - z$.
 	Hence $\partial_uF(z,U(z)) = 2U(z) - z$. Recalling that $\partial_uF(\rho,U(\rho))=1$, 
	 we get $U(\rho) = \frac{1+\rho}{2}$.
	In addition, $\partial^2_uF(z,u) = \exp(u) + D(z)\exp(u) = \partial_uF(z,u) + 1$.
	Therefore, $\partial^2_uF(\rho,U(\rho)) = 2$ and $\zeta = 1$.
	The asymptotics of $U$ and $U'$ follow.

        Regarding $U^\star$ , \cref{eq:Ustar} implies that $U^\star = ({1-\partial_uF(z,U(z))})^{-1}$. Similarly solving the system of equations \eqref{eq:Uevenodd} we get
	$U^{\even} = (1 - (\partial_uF(z,U(z)))^2)^{-1}$ and $U^{\odd} =  \partial_uF(z,U(z))U^{\even}$ .
	By the daffodil lemma \cite[Lemma IV.1, p.266]{Violet}, we have $|\partial_uF(z,U(z))|<1$ for $|z| \le \rho$
	and $z \ne \rho$. In particular, $\partial_uF(z,U(z))$ avoids the value $1$ and $-1$ for such $z$.
  Therefore   $U^\star$,  $U^{\even}$ and  $U^{\odd}$ are $\Delta$-analytic.
 The asymptotics of $U^\star$ follows 
	from the above results.
	Finally, since $\partial_uF(\rho,U(\rho)) =1$, we have 
	$U^{\even}\sim U^\odd$ when  $z$ tends to  $\rho$. 
	And, since $U^\star = U^{\even}+ U^\odd$, their asymptotics follow.
	\end{proof}

\subsection{Distribution of induced subtrees of uniform cotrees}
We take a uniform unlabeled canonical cotree $\bm t^{(n)}$ with $n$ leaves, \emph{i.e.} a uniform element of size $n$
in $\overline{\mathcal V}$. 
We also choose uniformly at random a $k$-tuple of distinct leaves of $\bm t^{(n)}$.
We denote by $\mathbf{t}^{(n)}_k$ the labeled cotree induced by the $k$ marked leaves.

\begin{proposition}\label{prop:proba_arbre_unlabeled}
	Let $k\geq 2$, and let $t_0$ be a labeled binary cotree with $k$ leaves. Then
	\begin{equation}\label{eq:proba_asymptotique_t0_unlabeled}
	\proba (\mathbf{t}^{(n)}_k = t_0) \underset{n\to+\infty}{\to} 
	\displaystyle{\frac{(k-1)!}{(2k-2)!}}. 
	\end{equation}
\end{proposition}

\begin{proof}
	We take a uniform random pair $(\bm T^{(n)},\bm a)$ of $\mathcal V$ of size $n$
	with a $k$-tuple of distinct leaves of $\bm T^{(n)}$, also chosen uniformly.
	We denote by $\mathbf{T}^{(n)}_k$ the cotree induced by the $k$ marked leaves.
	Since the forgetting map from $\mathcal V$ to $\overline{\mathcal V}$ is $n!$-to-$1$,
	$\bm T^{(n)}_k$ is distributed as $\bm t^{(n)}_k$.
    Hence, similarly to \cref{eq:probat0_quotient}, we have
	\[ \proba(\mathbf{t}^{(n)}_k=t_0) 
	= \proba(\mathbf{T}^{(n)}_k=t_0) \geq \frac{n![z^{n}] V_{t_0}(z)}{n \dots (n-k+1) \, n! [z^n] V(z)}.\]
	The inequality comes from the fact that $\mathcal{V}_{t_0}$ does not consist
	of all pairs in $\mathcal V$ with a $k$-tuple of marked leaves inducing $t_0$,
	but only of some of them (see the additional constraint in the second item of \cref{Def:V_t_0}).

	From \cref{th:V_t_0}, we have
	$$
	V_{t_0}=(U^\star)(2U+1-z)^{n_v} (U^\bullet)^k (U^\odd)^{n_=} (U^\even)^{n_{\neq}}.
	$$
	Recalling that $U^\bullet(z) = zU'(z)$, we use the asymptotics for $U,U',U^\star,U^\even, U^\odd$ (given in \cref{prop:Asympt_U}) 
	and furthermore the equalities $n_v=k-1$ and $n_= +n_{\neq}=k-2$ (which hold since $t_0$ is binary) to obtain
	\begin{align*}
	V_{t_0}(z)
	&\underset{z\to \rho}{\sim}  \frac {1}{2\beta} 2^{k-1} \left(\frac{\beta}{2} \cdot \rho\right)^k \left(\frac {1}{4\beta}\right)^{k-2} (\rho-z)^{-(k-1/2)}\\
	&\underset{z\to \rho}{\sim}  \frac {\beta \rho^k}{2^{2k-2}} (\rho-z)^{-(k-1/2)} = \frac {\beta\sqrt{\rho}}{2^{2k-2}} (1-\tfrac z \rho)^{-(k-1/2)}.
	\end{align*}
	By the transfer theorem (\cite[Corollary VI.1 p.392]{Violet}) we have 
$$
[z^{n}] V_{t_0}(z) \underset{n\to +\infty}{\sim} \frac{\beta\sqrt{\rho}}{2^{2k-2} \rho^{n}} \frac{n^{k-3/2}}{\Gamma(k-1/2)}
=  \beta  \frac{(k-1)!}{\sqrt{\pi}(2k-2)!} \frac{n^{k-3/2}}{\rho^{n-1/2}}
$$
	Besides, using $V(z)=2U(z)-z$, \cref{prop:Asympt_U}, and the transfer theorem as above, we have 
	$$
	n(n-1)\dots (n-k+1) [z^n] V(z) \underset{n\to +\infty}{\sim}
	n^k (-2\beta\sqrt{\rho})\frac{n^{-3/2}}{\rho^n \Gamma(-1/2)}
	\sim {\beta} \frac{n^{k-3/2}}{\rho^{n-1/2} \sqrt{\pi}}.
	$$
	Finally,
	$
	\liminf_{n\to\infty}\proba(\mathbf{t}^{(n)}_k=t_0) \geq \frac{(k-1)!}{(2k-2)!}
	$.
	To conclude, recall (as seen in the proof of \cref{lem:factorisation}) 
	that summing the right-hand-side over all labeled binary cotrees $t_0$ of size $k$ gives $1$,
	from which the proposition follows.
\end{proof}

\subsection{Proof of \cref{th:MainTheorem,th:DegreeRandomVertex} in the unlabeled case}\label{SousSec:PreuveNonEtiquete}
The argument is identical to the labeled case.
Recall that $\bm t^{(n)}$ is a uniform unlabeled canonical cotree of size $n$, so that $\Cograph(\bm t^{(n)})$ is a uniform unlabeled cograph of size $n$, \emph{i.e.} has the same ditribution as $\bm G_n^u$.
 Thus \cref{th:MainTheorem} follows from \cref{lem:factorisation} and \cref{prop:proba_arbre_unlabeled}, and
 \cref{th:DegreeRandomVertex} is then a consequence of  \cref{th:MainTheorem,prop:degree,prop:IntensityW12}.

\section{Vertex connectivity}
\label{sec:DegreeConnectivity}
A connected graph $G$ is said to be $k$-connected if it does not contain a set of $k - 1$ vertices whose removal disconnects the graph. 
The \emph{vertex connectivity} $\kappa(G)$ is defined as the 
largest $k$ such that $G$ is $k$-connected.

Throughout this section, $\bm G_n$ (resp. $\bm G^u_n$) 
is a uniform random labeled (resp. unlabeled) cograph of size $n$,
{\em conditioned to be connected}. 
The aim of this section is to prove that 
the random variable $\kappa(\bm G_n)$ (resp. $\kappa(\bm G^u_n)$) 
converges in distribution to a non-trivial random variable (without renormalizing).
The limiting distributions in the labeled and unlabeled cases are different.  
\medskip

A cograph $G$ (of size at least $2$) is connected if and only if the root of its canonical cotree is decorated by $\One$. 
(This implies that in both cases a uniform cograph of size $n$ is connected with probability $1/2$ for every $n$.)
Therefore, any connected cograph $G$ (of size at least $2$) can be uniquely decomposed as the join of $F_1,\dots, F_k$ where each $F_i$ is either a disconnected cograph or a one-vertex graph. 
Moreover, the cographs $F_i$ are those whose canonical cotrees are the fringe subtrees attached to the root of the canonical cotree of $G$. 
Throughout this section, we refer to the $F_i$'s as the \emph{components} of $G$. 
The following lemma, illustrated by \cref{fig:vertex_connectivity}, gives a simple characterization of $\kappa(G)$ when $G$ is a cograph.
\begin{lemma}\label{lem:CalculerKappa}
Let $G$ be a connected cograph which is not a complete graph.
Let $F_1,\dots, F_k$ be the components of $G$. It holds that 
$$
\kappa(G)= %
|G| - \max_{1\leq i \leq k} \{|F_i|\}.
$$
\end{lemma}
\begin{proof}
We reorder the components such that $|F_{1}|=\max_i|F_{i}|$. 
Because $G$ is not a complete graph, $F_{1}$ is not a one-vertex graph, and therefore is disconnected. 
Let us denote by $v_1,\dots,v_r$ the vertices of $F_{2}\cup F_{3} \cup \dots \cup F_{k}$. We have to prove that $\kappa(G)=r$.

\noindent{\bf Proof of $\kappa(G)\leq r$.} If we remove all vertices $v_1,\dots,v_r$ then we are left with $F_1$ which is disconnected.\\
\noindent{\bf Proof of $\kappa(G)\geq r$.} If we remove only $r-1$ vertices then there remains at least one $v_j$ among  $v_1,\dots,v_r$. 
Let us denote by $F_i$ the component of $v_j$.
There also remains at least a vertex $v \notin F_i$ (or $|F_i|$ would be larger than $|F_1|$). 
Consequently, $v$ and $v_j$ are connected by an edge, and every remaining vertex is connected to $v_j$ (when not in $F_i$) or to $v$ (when not in the component containing $v$), 
so that $G$ remains connected. 
Therefore we must remove at least $r$ points to disconnect $G$.
\end{proof}

\begin{figure}[htbp]
\begin{center}
\includegraphics[width=11cm]{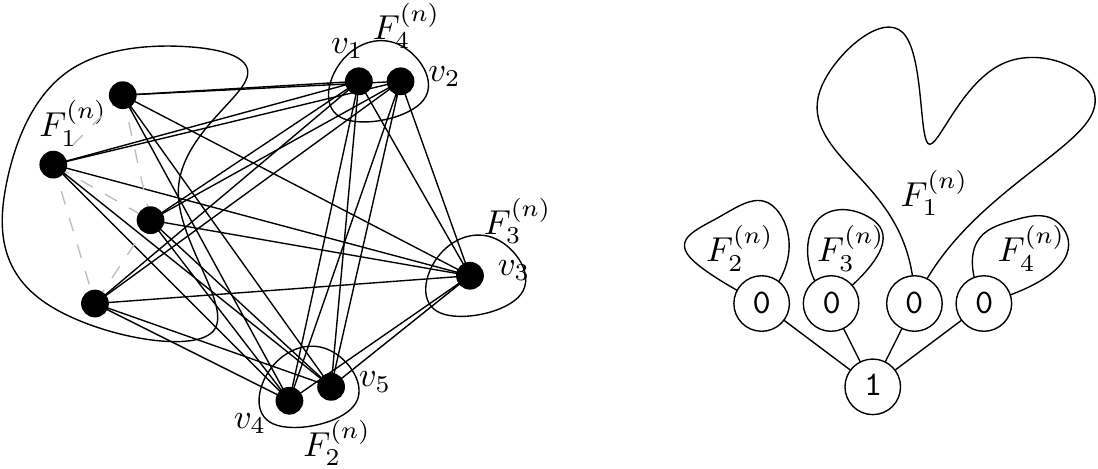}
\caption{A connected cograph and the corresponding cotree. The connectivity degree of this graph is $|F_{2}|+|F_{3}|+|F_{4}|=2+2+1=5$. \label{fig:vertex_connectivity}}
\end{center}
\end{figure}

\begin{theorem}\label{thm:DegreeConnectivity}
Let $M(z)$ (resp. $V(z)$) be the exponential (resp. ordinary) generating series of labeled (resp. unlabeled) cographs.
Their respective radii of convergence are $\rho=2\log(2)-1$ and $\rho_u \approx 0.2808$.
For $\formerell \geq 1$, set
$$
\pi_\formerell= \rho^\formerell [z^{\formerell}]M(z), \qquad \pi_\formerell^u = \rho_u^\formerell [z^{\formerell}] V(z).
$$
Then $(\pi_\formerell)_{\formerell\geq 1}$ and $(\pi_\formerell^u)_{\formerell\geq 1}$ are probability distributions 
and, for every fixed $\formerell \ge 1$,
\begin{equation}
\label{eq:limit_pi_l}
\mathbb{P}(\kappa(\bm G_n)=\formerell) \underset{n\to +\infty}{\to}\pi_\formerell,
\qquad \mathbb{P}(\kappa(\bm G^u_n)=\formerell) \underset{n\to +\infty}{\to}\pi^u_\formerell.
\end{equation}
\end{theorem}

\begin{remark}
Readers acquainted with Boltzmann samplers
may note that $(\pi_\formerell)_{\formerell\geq 1}$ and $(\pi_\formerell^u)_{\formerell\geq 1}$
are distributions of sizes of Boltzmann-distributed random labeled and unlabeled cographs, respectively.
The Boltzmann parameters are chosen to be the radii of convergence.
We do not have a direct explanation of this fact.
\end{remark}
\begin{proof}
Recall from \cref{sec:proofLabeled,sec:unlabeled} that $M(z)=2L(z)-z$ and $V(z)=2U(z)-z$.
It follows from \cref{prop:Asympt_S_Seven,prop:Asympt_U} that $\rho=2\log(2)-1$ and $\rho_u \approx 0.2808$ are their respective radii of convergence.
We first prove that  $(\pi_\formerell)$ (resp. $(\pi^u_\formerell)$) sum to one: 
\begin{align*}
\sum_{\formerell \geq 1} \pi_\formerell &= \sum_{\formerell\geq 1} \rho^\formerell [z^\formerell]M(z)=M(\rho)=2L(\rho)-\rho=1,\\
\sum_{\formerell \geq 1} \pi^u_\formerell &= \sum_{\formerell\geq 1} \rho_u^\formerell [z^\formerell]V(z)=V(\rho_u)=2U(\rho_u)-\rho_u = 1,
\end{align*}
using \cref{prop:Asympt_S_Seven,prop:Asympt_U} for the last equalities. 
\smallskip

For the remaining of the proof, we fix $\formerell \geq 1$.
In the labeled case, let $\bm T_n$ be the canonical cotree of $\bm G_n$.
Since $\bm G_n$ is conditioned to be connected, $\bm T_n$ is a uniform labeled canonical cotree of size $n$
conditioned to have root decoration $\One$.
Forgetting the decoration, we can see it as a uniform random element of size $n$ in $\mathcal L$.

Let $n > 2\formerell$. As the components of $\bm G_n$ correspond to the subtrees attached to the root of $\bm T_n$, 
using \cref{lem:CalculerKappa} we have ${\kappa}(\bm G_n)=\formerell$
if and only if $\bm T_n$ is composed of a tree of $\mathcal{L}$ of size $n-\formerell$
and $k\geq 1$ trees of $\mathcal{L}$ of total size $\formerell$, all attached to the root.
Since $n> 2\formerell$, the fringe subtree of size $n- \formerell$ is uniquely defined, and there is only one such decomposition.
Therefore, %
for every fixed $\formerell \ge 1$ and $n>2j$, we have
\[
\mathbb{P}({\kappa}(\bm G_n)=\formerell) 
=\frac{[z^{n-\formerell}]L(z) \, [z^{\formerell}] \! \left(e^{L(z)}-1\right)}{[z^{n}]L(z)}.
\]
From \cref{prop:Asympt_S_Seven},
the series $L(z)$ has radius of convergence $\rho$,
is $\Delta$-analytic and has a singular expansion amenable to singularity analysis.
Thus, the transfer theorem ensures that $\frac{[z^{n-\formerell}]L(z)}{[z^{n}]L(z)}$ tends to $\rho^\formerell$,
so that
$$
\mathbb{P}({\kappa}(\bm G_n)=\formerell) \underset{n\to +\infty}{\to}  \rho^\formerell \, [z^{\formerell}] \! \left(e^{L(z)}-1\right) = \pi_\formerell,
$$
where we used $M(z)=e^{L(z)}-1$ (see \cref{eq:Lien_T_expS}).

In the unlabeled case, let $\bm T_n^u$ be the canonical cotree of $\bm G_n^u$. 
Like in the labeled case, forgetting the decoration, it is a uniform element of $\overline{\mathcal U}$ of size $n$.
Let $n >2\formerell$. We have ${\kappa}(\bm G_n^u)=\formerell$
if and only if $\bm T_n^u$  has a fringe subtree of size $n-\formerell$ at the root. 

Let us count the number of trees of $\overline{\mathcal U}$ of size $n$ that have a fringe subtree of size $n-\formerell$ at the root.  Since $n-\formerell>n/2$, there must be exactly one such fringe subtree, and there are $[z^{n-\formerell}]U(z)$ choices for it. Removing it, the rest of the tree contains $\formerell$ leaves, and is either a tree of $\overline{\mathcal U}$ of size $\geq 2$ (if the root still has degree at least 2),
or a tree formed by a root and a single tree of $\overline{\mathcal U}$ attached to it.
So the number of choices for the rest is $[z^{\formerell}](2U(z)-z)$.
We deduce that for $j\geq 1$ and $n>2j$,
\[
\mathbb{P}({\kappa}(\bm G_n^u)=\formerell) 
= \frac{[z^{n-\formerell}]U(z) \, [z^{\formerell}](2U(z)-z)}{[z^{n}]U(z)}.
\]
From \cref{prop:Asympt_U}, %
the series $U(z)$ has radius of convergence $\rho_u$,
is $\Delta$-analytic and has a singular expansion amenable to singularity analysis.
The transfer theorem ensures that $\frac{[z^{n-\formerell}]U(z)}{[z^{n}]U(z)}$ tends to $\rho_u^\formerell$,
so that
\[
\mathbb{P}({\kappa}(\bm G_n^u)=\formerell) \underset{n\to +\infty}{\to} \rho_u^\formerell \, [z^{\formerell}] (2U(z)-z)= \pi^u_\formerell
\]
where we used $V(z) = 2U(z)-z$.%
\end{proof}

\begin{remark}
In the labeled case, we could have used \cref{lem:CalculerKappa} and local limit results for trees
instead of the generating series approach above.
Indeed, the canonical cotree of $\bm G_n$ (without its decorations) is distributed as a Galton-Watson
tree with an appropriate offspring distribution conditioned on having $n$ leaves.
Such conditioned Galton-Watson trees converge in the local sense near the root 
towards a Kesten's tree \cite[Section 2.3.13]{AD15}.
Since Kesten's trees have a unique infinite path from the root,
this convergence implies the convergence (without renormalization) of 
the sizes of all components of $\bm G_n$ but the largest one.
Therefore the sum $\kappa(\bm G_n)$ of these sizes also converges (without renormalization);
the limit can be computed (at least in principle) using the description of Kesten's trees.

In the unlabeled case, the canonical cotree of $\bm G_n^u$ (without its decorations) 
belongs to the family of random {\em P\'olya} trees.
Such trees are {\em not} conditioned Galton-Watson trees.
For scaling limits, it has been proven they can be approximated by conditioned Galton-Watson trees
and hence converge under suitable conditions to the Brownian Continuum Random Tree \cite{PanaStufler},
but we are not aware of any local limit result for such trees.
\end{remark}

\subsection*{Acknowledgments}
MB is partially supported by the Swiss National Science Foundation, under grants number 200021\_172536 and PCEFP2\_186872.

\end{document}